\documentclass{article}[12pt]
\usepackage[utf8]{inputenc}
\usepackage{BAstyle}

\usepackage[margin=1.5in]{geometry}

\newtheorem*{MB*}{Malle-Bhargava Prediction}

\title{Statistics of the First Galois Cohomology Group:\\ A Refinement of Malle's Conjecture}

\author{Brandon Alberts}

\date{}

\begin{document}

\maketitle

\abstract{Malle proposed a conjecture for counting the number of $G$-extensions $L/K$ with discriminant bounded above by $X$, denoted $N(K,G;X)$, where $G$ is a fixed transitive subgroup $G\subset S_n$ and $X$ tends towards infinity. We introduce a refinement of Malle's conjecture, if $G$ is a group with a nontrivial Galois action then we consider the set of crossed homomorphisms in $Z^1(K,G)$ (or equivalently $1$-coclasses in $H^1(K,G)$) with bounded discriminant. This has a natural interpretation given by counting $G$-extensions $F/L$ for some fixed $L$ and prescribed extension class $F/L/K$.

If $T$ is an abelian group with any Galois action, we compute the asymptotic growth rate of this refined counting function for $Z^1(K,T)$ (and equivalently for $H^1(K,T)$) and show that it is a natural generalization of Malle's conjecture. The proof technique is in essence an application of a theorem of Wiles on generalized Selmer groups, and additionally gives the asymptotic main term when restricted to certain local behaviors. As a consequence, whenever the inverse Galois problem is solved for $G\subset S_n$ over $K$ and $G$ has an abelian normal subgroup $T\normal G$ we prove a nontrivial lower bound for $N(K,G;X)$ given by a nonzero power of $X$ times a power of $\log X$. For many groups, including many solvable groups, these are the first known nontrivial lower bounds. These bounds prove Malle's predicted lower bounds for a large family of groups, and for an infinite subfamily they generalize Kl\"uners' counter example to Malle's conjecture and verify the corrected lower bounds predicted by T\"urkelli.}

\newpage

\tableofcontents

\newpage

\section{Introduction}

Number field counting problems began by asking questions about how many number fields there are with bounded discriminant. In the study of this topic, the problem naturally partitioned into counting number fields with prescribed Galois group. Malle \cite{malle2002,malle2004} collected this problem together under the roof of a single conjecture. Let $K$ be a number field and $G_K=\Gal(\overline{K}/K)$ be its absolute Galois group throughout the paper. If $L/K$ is a degree $n$ extension, we refer to the Galois group $\Gal(L/K)\subset S_n$ as the Galois group of the Galois closure of $L/K$ together with the action permuting the $n$ embeddings of $L$ into the algebraic closure $\overline{K}$. If $G\subset S_n$ is a transitive permutation group, we may ask how many degree $n$ extensions $L/K$ there are with $\Gal(L/K)\cong G$ and bounded discriminant, i.e. what is the size of the counting function
\[
N'(K,G;X):= \#\{L/K : [L:K]=n,\Gal(L/K)\cong G, \mathcal{N}_{K/\Q}(\disc(L/K))<X\}.
\]
Malle gave theoretical evidence in \cite{malle2004} suggesting how this function should grow asymptotically as $X$ tends to infinity. This is often referred to as the ``Strong Form" of Malle's conjecture.

\begin{conjecture}[Strong Form of Malle's Conjecture]
Let $G\subset S_n$ be a transitive subgroup and define the class function $\ind:G\rightarrow \Z$ by $\ind(g)= n- \#\{\text{orbits of }g\}$. Then
\[
N'(K,G;X)\sim c(K,G)X^{1/a(G)}(\log X)^{b(K,G)-1},
\]
where $c(K,G)>0$, $a(G) = \min_{g\ne 1} \ind(g)$ and
\[
b(K,G) = \#\left(\{\text{conjugacy class }C\subset G : \ind(C) = a(G)\} / \chi\right)
\]
is the number of orbits under the action by the cyclotomic character $\chi:G_K\rightarrow \hat{\Z}$ on the set of conjugacy classes, where the action is given by $\sigma.g=g^{\chi(\sigma)}$.
\end{conjecture}

We remark that, for any transitive subgroup $G\subset S_n$, there are two different, yet equivalent, conventions for counting number fields with Galois group $G\subset S_n$. This is described by the correspondence below:
\begin{align*}
\left\{\substack{\ds L/K \text{ degree }n\text{ with }\Gal(\widetilde{L}/K)\cong G\\\ds\text{ordered by }\disc(L/K)}\right\} \leftrightarrow\left\{\substack{\ds \widetilde{L}/K \text{ Galois with }\Gal(\widetilde{L}/K)\cong G\\\ds\text{ordered by the discriminant}\\\ds\text{of a subfield fixed}\\\ds\text{by a point stabilizer}}\right\}\,,
\end{align*}
where the correspondence is $\frac{|G|}{n}$-to-$1$. Thus, up to a constant, counting number fields from either perspective amounts to an equivalent result. A majority of authors working in the area of number field counting pick one perspective and stick with it, and we will do the same. It will be convenient to deal only with Galois extensions, so we will work on the right hand side of this correspondence. All field extensions will be Galois unless stated otherwise, and a $\mathbf{G}$\textbf{-extension} $L/K$ will refer to a Galois extension with Galois group $G$ with discriminant ordering given by $\disc(L^{{\rm Stab}_G(1)}/K)$. We define
\[
N(K,G;X) :=\#\left\{L/K \text{ Galois}: \Gal(L/K)\cong G,\ \disc(L^{{\rm Stab}_G(1)}/K)<X\right\},
\]
and knowing that this counting function is only off from Malle's original formulation by a constant factor implies that Malle's conjecture for $N'(K,G;X)$ is equivalent to the same statement for $N(K,G;X)$.

The strong form of Malle's conjecture is known to be true in the following cases:
\begin{itemize}
\item{$G$ abelian was proven by Wright \cite{wright1989},}

\item{$G=S_n$ for $n=3$ by Datskovsky-Wright \cite{datskovsky-wright1988} and $n=4,5$ by Bhargava-Shankar-Wang \cite{bhargava-shankar-wang2015},}

\item{$S_3\subset S_6$ by Bhargava-Wood \cite{bhargava-wood2007},}

\item{$D_4\subset S_4$ over $K=\Q$ by Cohen-Diaz y Diaz-Olivier \cite{cohen-diaz-y-diaz-olivier2002},}

\item{$Q_{4m}\subset S_{4m}$ the generalized quaternion group of order $4m$ by Kl\"uners \cite{klunersHab2005},}

\item{$C_2\wr H$ for many groups $H$ by Kl\"uners \cite{kluners2012},}

\item{$S_n\times A$ for $n=3,4,5$ and $|A|$ coprime to $2,6,30$ respectively by Wang \cite{jwang2017},}

\item{$D_4\subset S_8$ in an upcoming preprint by Shankar-Varma \cite{shankar-varma2019},}

\item{$T\wr B$ for $T$ either abelian or $S_3$ and $B$ any group such that $N(K,B;X)$ grows sufficiently slowly in an upcoming preprint of Lemke Oliver-Wang-Wood \cite{lemke-oliver-jwang-wood2019}.}
\end{itemize}
Unfortunately, Malle's conjecture is not true in general. Kl\"uners showed that for $G=C_3\wr C_2$ and $K=\Q$ Malle's predicted log factor is too small \cite{kluners2005}.

The power of $X$ factor is generally believed to be correct, which leads many authors to consider a weaker version of this conjecture proposed in Malle's earlier paper \cite{malle2002}:

\begin{conjecture}[Weak Form of Malle's Conjecture]
Let $G\subset S_n$ be a transitive subgroup and define the class function $\ind:G\rightarrow \Z$ by $\ind(g)= n- \#\{\text{orbits of }g\}$. Then
\[
X^{1/a(G)}\ll N(K,G;X) \ll X^{1/a(G)+\epsilon},
\]
where $a(G) = \min_{g\ne 1} \ind(g)$.
\end{conjecture}

The weak form has no known counterexamples, and more information is known for several different groups $G$.
\begin{itemize}
\item{The weak form holds for $G$ nilpotent in the regular representation by Kl\"uners-Malle \cite{kluners-malle2004},}

\item{The upper bound holds for $G$ a $p$-group by Kl\"uners-Malle \cite{kluners-malle2004},}

\item{The lower bound holds for $G=D_p$ for $p$ an odd prime in the degree $p$ and $2p$ representations over $K=\Q$ by Kl\"uners \cite{kluners2006}, as well as the upper bound conditional on Cohen-Lenstra heuristics,}

\item{The upper bound holds for $G$ nilpotent in any representation by the author in \cite{alberts2020},}

\item{The upper bound holds for $G$ solvable in any representation conditional on the $\ell$-torsion conjecture for class groups by the author in \cite{alberts2020}.}

\end{itemize}

There are nontrivial upper bounds for all groups $G$ which are not believed to be sharp. Such examples can be found in papers by the author \cite{alberts2020}, Dummit \cite{dummit2018}, Ellenberg-Venkatesh \cite{ellenberg-venkatesh2005}, and Schmidt \cite{schmidt1995}. Nontrivial lower bounds tend to be rarer in the literature, and are not known for every group (such a result would solve the inverse Galois problem!). Besides the groups listed above for which Malle's predicted lower bound is known, the author is only aware of the following nontrivial lower bounds:
\begin{itemize}
\item{If the inverse Galois problem is solved for $G$ over $K$ and $Z\le G$ is a central subgroup, then $N(K,G;X) \gg X^{\frac{\ell}{n(\ell-1)}}$ for $\ell$ the smallest prime dividing $|Z|$ by Kl\"uners-Malle \cite{kluners-malle2004},}

\item{$N(\Q,S_n;X)\gg X^{1/n}$ by Malle \cite{malle2002},}

\item{$N(K,S_n;X)\gg X^{\frac{1}{2}-\frac{1}{n^2}}$ by Ellenberg-Venkatesh \cite{ellenberg-venkatesh2005},}

\item{$N(\Q,S_n;X)\gg X^{\frac{1}{2}+\frac{1}{n}}$ by Bhargava-Shankar-Wang \cite{bhargava-shankar-wang2016},}

\item{$N(\Q,A_4;X)\gg X^{\frac{1}{2}}$ by Baily \cite{baily1980},}

\item{$N(\Q,A_n;X)\gg X^{\frac{n! - 2}{n!(4n-4)}}$ by Pierce-Turnage-Butterbaugh-Wood \cite{pierce-turnage-butterbaugh-wood2017},}

\item{$N(\Q,G;X)\gg X^{\frac{|G|-1}{d|G|(2n-2)}}$ whenever there exists a regular polynomial in $\Q[X,T_1,...,T_s]$ with Galois group $G$ and degree $\le d$ in the $T$ variables by Pierce-Turnage-Butterbaugh-Wood \cite{pierce-turnage-butterbaugh-wood2017}.}
\end{itemize}

Kl\"uners-Malle show that their bound realizes Malle's predicted lower bounds if $G$ is a nilpotent group in the regular representation, and comment that it realizes the predicted lower bound in some other cases (for example, any group $C_2\times H$ in the regular representation).

One of the modern approaches to Malle's conjecture is inductively counting extensions. These methods are used to prove Malle's conjecture in the only large families of nonabelian groups for which the conjecture is known, namely $C_2\wr H$, $S_n\times A$, and $T\wr B$ for $N(K,B;X)$ growing slow enough, as well as for both transitive representations of $D_4$. Fix a finite group $G$, and say we want to count $G$-extensions $F/K$ ordered by some invariant (recall our convention that a $G$-extension is always Galois, and special ordering s correspond to different transitive representations of $G$). If $G$ is not a simple group, we could potentially break down this counting problem into two separate counting problems. An upcoming preprint by Lemke Oliver-Wang-Wood \cite{lemke-oliver-jwang-wood2019} formally introduces this approach, and we adopt their very intuitive notation. Suppose $T\normal G$ is a normal subgroup with quotient group $G/T = B$. Any $G$-extension $F/K$ decomposes into a tower of fields:
\begin{align}\label{tower}
\begin{tikzpicture}[node distance =2cm, auto]
	\node(K) {$K$};
	\node(L) [above right of =K] {$L$};
	\node(F) [above left of=L] {$F$};
	\draw[-] (L) to node {$B$} (K);
	\draw[-] (F) to node {$T$} (L);
	\draw[-] (K) to node {$G$} (F);
\end{tikzpicture}
\end{align}
We may think of ``$T$" as standing for ``top extension" and ``$B$" as standing for ``bottom extension" to help us keep track of the notation. Inductive approaches to Malle's conjecture and number field counting involve first counting the number of $T$-extensions $F/L$, then summing over $B$-extensions $L/K$. Written explicitly, this strategy can be expressed as a two step process:

\textbf{Step 1:} For a fixed intermediate $B$-extension $L/K$, determine the asymptotic growth of the function
\[
N(L/K,T\normal G;X) := \#\{F/K : \Gal(F/K) \cong G,\ F^T = L,\ \mathcal{N}_{K/\Q}(\disc(F/K)) < X\}\,.
\]
This function counts the number of towers $F/L/K$ in the form of (\ref{tower}) with a prescribed choice of $L$.

\textbf{Step 2:} We take a sum over all the choices for the intermediate extension $L/K$, which satisfies
\[
N(K,G;X) = \sum_{\substack{L/K\\ \Gal(L/K)\cong B}} N(L/K,T\normal G;X)\,.
\]
This sum counts all possible towers $F/L/K$ in the form of (\ref{tower}).

For example, this is the approach taken by Wang \cite{jwang2017} for $G=S_n\times A$, where she takes $T=A$ and $B=S_n$. Each step comes with a major obstacle:

\textbf{Obstacle for Step 1:} The step 1 counting function looks just like counting extensions $F/L$ with Galois group $T$, i.e. Malle's counting function $N(L,T;X)$, except we also need to control for the total Galois group of the tower $\Gal(F/K)$. It is not always clear which $T$-extensions $F/L$ have total Galois group $G$ over $K$. In a work in progress, Lemke Oliver, Wang, and Wood \cite{lemke-oliver-jwang-wood2019} make this approach work for certain groups by additionally considering local behavior at finitely many places. For example, if $G=T\wr B$, then a $T$-extension of $L$ can be forced to have total Galois group $G$ by considering extensions with $I_{\mathfrak{p}_1}\ne 1$ and $I_{\mathfrak{p}_i}=1$ for $i\ne 1$ among all places $\mathfrak{p}_i$ of $L$ dividing a fixed place $p$ of $K$.

\textbf{Obstacle for Step 2:} Just because we know the asymptotic main term for each counting function in step 1 does not mean that we can necessarily add them all up. There are infinitely many choices for intermediate $B$-extensions $L/K$, and it is possible for a sum of infinitely many error terms to become larger than the main term. In order to make this approach work, we need to prove step 1 uniform in the choice of field $L/K$, i.e. we need to understand explicitly how the size of the error in our solution to step 1 depends on the $B$-extension $L/K$.

For the purposes of this paper, we will focus on step 1. The approach utilized by \cite{lemke-oliver-jwang-wood2019} becomes trickier when $G$ is not as nice of an extension of $B$ by $T$, and it is not necessarily clear that we can choose local conditions to force the total Galois group we want in all cases. The issue starts with an embedding problem: suppose $\gamma:G_K\rightarrow \Gal(L/K)\cong B$ is the quotient map defining the bottom extension. When does there exist a lift $\widetilde{\gamma}:G_K\rightarrow G$ such that the diagram
\[
\begin{tikzcd}
{}&&&G_K\dar{\gamma} \arrow[dashed,swap]{dl}{\widetilde{\gamma}}\\
1 \rar & T \rar & G \rar & B \rar & 1
\end{tikzcd}
\]
commutes? If $T$ is a central subgroup of $G$, then this embedding problem has a solution if and only if the corresponding local embedding problems have solutions (for a good reference, see Serre's Topics in Galois Theory \cite{serre2008}). In this case, this style of approach has led to new results in the study of nonabelian Cohen-Lenstra moments for nilpotent groups $G$ such as in joint work of the author with Klys \cite{alberts-klys2017}. When $T$ is not central, this problem becomes much harder and much less is known.

One of the key ideas in this paper is to push the issues of solving the embedding problem to the side. Fix an extension $L/K$ and suppose we \emph{already know} that there exists at least one extension $F/L/K$ such that $\Gal(F/K)\cong G$ and $F^T=L$, which is given by the surjective homomorphism $\pi:G_K\twoheadrightarrow G$ under the Galois correspondence. We will discuss how to count the number of such towers $F/L/K$ with $L$ fixed when we know there is at least one, which isolates the obstructions arising from the embedding problem away from the analytic and statistical results.

\begin{lemma}\label{lem:Z1}
Fix a homomorphism $\pi:G_K\rightarrow G$, and let $T(\pi)$ denote the group $T$ with the Galois action $x.t=\pi(x)  t  \pi(x)^{-1}$. Then there is a bijection
\[
Z^1(G_K,T(\pi)) \leftrightarrow \{\pi'\in \Hom(G_K,G) : \pi(x)T=\pi'(x)T\text{ for all }x\in G_K\}
\]
given by the map $f\mapsto f*\pi$, where $(f*\pi)(x)=f(x)\pi(x)$. We will often omit the $\pi$ and just write $T$ when the action is clear from context.
\end{lemma}

The surjective homomorphisms on the right-hand side are exactly the towers $F/L/K$ counted by the step 1 counting function $N(L/K,T\normal G;X)$ under the Galois correspondence. So, up to issues of surjectivity, these towers are in bijection with crossed homomorphisms $Z^1(G_K,T(\pi))$.

We define the $\pi$-discriminant on crossed homomorphisms to be the pull-back of the usual discriminant of the fixed field of a point stabilizer contained in the field corresponding to a surjective homomorphism under the bijection between towers and $Z^1(G_K,T(\pi))$, which is given by
\[
\disc_{\pi}(f) = \disc(f*\pi)\,.
\]
This must be extended in an appropriate way to $f$ for which $f*\pi$ is not surjective, which we do by defining $\disc(f*\pi)$ to be the discriminant of the $G$-\'etale algebra associated to $f*\pi$. We will elaborate on discriminants of nonsurjective elements in Section \ref{sec:towers}.


This tells us that counting towers $F/L/K$ is essentially the same as counting crossed homomorphisms
\begin{align*}
N(L/K,T\normal G;X) = \text{``surjective" elements of }Z^1\left(K,T(\pi);X\right),
\end{align*}
where $\pi$ is a surjective map corresponding to one such tower $F/L/K$ and $Z^1(K,T(\pi);X)$ denotes the number of crossed homomorphisms with $\pi$-discriminant bounded by $X$. Here we take ``surjective" to mean that $f$ corresponds to a surjective solution to the embedding problem $G_K\rightarrow G$ under the map $f\mapsto f*\pi$.

We can think of this as a direct generalization to classical number field counting problems and Malle's conjecture, where Malle predicts the growth of
\begin{align*}
N(K,T;X)&=\text{surjective elements of }\Hom(G_K,T;X)\,,
\end{align*}
noting that if $T=T(1)$ has the trivial Galois action then $Z^1(K,T) = \Hom(G_K,T)$.

This suggests that step 1 is an interesting question in its own right, as a natural generalization of Malle's conjecture:
\begin{question}
How do the number of ``surjective'' elements of $Z^1(K,T(\pi);X)$ grow as $X$ tends towards $\infty$, i.e. the number of $f\in Z^1(K,T(\pi);X)$ such that $f*\pi$ is surjective?
\end{question}

We extrapolate the heuristic justifications of Malle's conjecture to make a prediction for this behavior. In particular, we prove that the Malle-Bhargava principle \cite{bhargava2007,wood2017} gives the following prediction to this generalized question:

\begin{MB*}
Fix $G\subset S_n$, $T\normal G$, and $\pi:G_K\rightarrow G$ a homomorphism with $\pi(G_K)T=G$. Define the class function $\ind(g)=n-\#\{\text{orbits of }g\}$. Then
\[
N(L/K,T\normal G;X) \sim c(K,T) X^{1/a(T)}(\log X)^{b(K,T(\pi))-1}.
\]
where $c(K,T(\pi))>0$, $a(T) = \min_{t\in T-\{1\}}\ind(t)$ and
\[
b(K,T(\pi)) = \#\left(\{\text{conjugacy class }C\subset T : \ind(C)=a(T)\}/\pi*\chi^{-1}\right)
\]
is the number of orbits under the composite action given by $\pi$ and the cyclotomic character $\chi:G_K\rightarrow \hat{\Z}$ on the set of conjugacy classes, where the action is given by $\sigma.g=\pi(\sigma)g^{\chi(\sigma)^{-1}}\pi(\sigma)^{-1}$.
\end{MB*}
The invariants $a(T)$ and $b(K,T(\pi))$ exactly correspond to Malle's predicted invariants, where we make sure to account for the extra ``conjugates" under the Galois action by $\pi$. The condition that $\pi(G_K)T=G$ is necessary for the existence of surjective crossed homomorphisms, otherwise the resulting counting function is identically $0$. In particular, the case $\pi=1$, $T=G$ reproduces Malle's original predictions. Of course, we know Malle's conjecture is not true as stated, there are known counter examples such as $G=C_3\wr C_2$ by Kl\"uners \cite{kluners2005}. We should be hesitant to make new, wider reaching conjectures before fixing Malle's original conjecture.

In this paper we begin the process of justifying a prediction of this form, discussing what it would take to be internally consistent and consistent with Malle's conjecture as well as proposed corrections of Malle's conjecture. In certain cases where Malle's conjecture fails, such as $G=C_3 \wr C_2$, this refinement highlights the issues more clearly and suggests what we might want to change in order to repair Malle's conjecture. We will compare these insights to Turkelli's proposed correction to Malle's conjecture \cite{turkelli2015}.

To lend more credence to the idea that something of this form should be true, we prove it when $T$ is abelian. The first infinite family of groups for which Malle's conjecture was verified is the family of abelian groups, proven by Wright \cite{wright1989}, and we see Wright's result as a special case of the following theorems for abelian groups with arbitrary Galois actions:
\begin{theorem}\label{thm:MalleZ1}
Fix $G\subset S_n$ a transitive subgroup, $T\normal G$ an abelian normal subgroup, and $\pi:G_K\rightarrow G$ a homomorphism inducing a Galois action on $T$ by conjugtion. Then
\[
|Z^1(K,T(\pi);X)| \asymp X^{1/a(T)}(\log X)^{b(K,T(\pi))-1},
\]
where
\[
a(T) = \min_{t\in T-\{1\}} \ind(t),
\]
and
\[
b(K,T(\pi)) = \#\{t\in T : \ind(t)=a(T)\}/\pi*\chi^{-1},
\]
i.e. the number of orbits under the action $x.t=\pi(x)t^{\chi(x)^{-1}}\pi(x)^{-1}$. Here we write $f(X) \asymp g(X)$ to mean that there exist positive constants $c_1$ and $c_2$ such that
\[
c_1 g(X) \le f(X) \le c_2 g(X)
\]
for all sufficiently larger values of $X$.
\end{theorem}

By carefully applying an inclusion-exclusion argument, we can sieve to surjective maps in order to prove the following:

\begin{corollary}\label{cor:alpha}
Fix $G\subset S_n$ a transitive subgroup and $T\normal G$ an abelian normal subgroup with $B:=G/T$. If $L/K$ is a fixed $B$-extension (i.e Galois with $\Gal(L/K)\cong B$) and there exists a (not necessarily surjective) homomorphism $\pi:G_K\rightarrow G$ such that the fixed field of $T\ker \pi$ is $L$, then
\[
N(L/K,T\normal G;X) \asymp X^{1/a(T)}(\log X)^{b(K,T(\pi))-1}.
\]
In particular, the existence of a non-surjective solution to the embedding problem implies the existence of a surjective solution.
\end{corollary}

We actually prove a more general results where we are allowed to restrict to certain local behaviors, such as requiring the crossed homorphisms to be unramified at a fixed finite set of places, as well as counting under different orderings of the crossed homomorphisms such as taking the norm of the product of ramified primes to be bounded by $X$. Theorem \ref{thm:MalleZ1} is proven using results on group cohomology, and in the process we will show it is equivalent to analogous asymptotic results on the number of elements of $H^1(K,T(\pi))$ with bounded discriminant. We also prove sufficient conditions for 100\% of $1$-coclasses to be surjective, so that in these special cases
\[
N(L/K,T\normal G;X) \sim |Z^1(K,T(\pi);X)|\,.
\]

It is an artifact of the methods used in this paper that we do not achieve the main term on the nose, but instead just get the order of growth up to a bounded function (the difference between $\sim$ and $\asymp$). We discuss this further in Section \ref{sec:counting}, and the more general results prove in Sections \ref{sec:AsymptoticWiles} and \ref{sec:counting} do give the main term on the nose under slightly nicer orderings and/or restricted local conditions. One such example is the ordering given by the product of ramified primes which are unramified in $\pi$
\[
\ram_{\pi}(f) = \prod_{\substack{p:f(I_p)\ne 1\\ \pi(I_p)=1}} p,
\]
for which
\begin{corollary}\label{cor:ram}
Fix $G$ a finite group and $T\normal G$ an abelian normal subgroup with $B:=G/T$ and $\pi:G_K\rightarrow G$ such that the fixed field of $T\pi(G_K)=G$, then there exists a positive constant $c$ such that
\begin{align*}
cX(\log X)^{b-1} &\sim \#\{f\in Z^1(K,T(\pi)) : \mathcal{N}_{K/\Q}(\ram_\pi(f)) < X\}\\
&\sim\#\{f\in Z^1(K,T(\pi)) : f*\pi\text{ surjective, }\mathcal{N}_{K/\Q}(\ram_\pi(f)) < X\}\,,
\end{align*}
where
\[
b = \#\left(\{T-\{1\}\}/\pi*\chi^{-1}\right).
\]
\end{corollary}

These results solve step 1 of the inductive method when $T$ is abelian in the case of nice orders, modulo the embedding problem, and give a sharp asymptotic growth rate for step 1 in the $\pi$-discriminant ordering. This opens the door to applying the process outlined in \cite{lemke-oliver-jwang-wood2019} to many more groups, where it now suffices to consider step 2 to get the asymptotic growth rate for many nonabelian groups $G$. Future work of the author on Malle's conjecture will involve combining the methods in \cite{lemke-oliver-jwang-wood2019} with the results of this paper to prove results uniform in the base field with the goal of completing step 2 and proving the strong form of Malle's conjecture for many more new groups. The author is also working on a greater generalization of the methods in this paper to proof the main term for Theorem \ref{thm:MalleZ1} and Corollary \ref{cor:alpha}.

Even without answering any questions on uniformity we can use the results of this paper on step 1 as a lower bound for the number of $G$-extensions in Malle's conjecture proper:

\begin{corollary}\label{cor:lowerbound}
Fix $G\subset S_n$ a transitive subgroup, $T\normal G$ an abelian normal subgroup, and suppose there exists at least one $G$-extension $F/K$. Then for the corresponding $\pi:G_K\twoheadrightarrow \Gal(F/K)$
\[
N(K,G;X) \gg X^{1/a(T)}(\log X)^{b(K,T(\pi))-1}.
\]
In particular, we have the following special cases:
\begin{itemize}
\item[(i)]{If there exists $t\in T$ with $\ind(t)=a(G)$, then
\[
N(K,G;X) \gg X^{1/a(G)},
\]
which satisfies Malle's predicted weak lower bound.
}

\item[(ii)]{If $\{g\in G: \ind(g)=a(G)\}\subset T$, then
\[
N(K,G;X) \gg X^{1/a(G)}(\log X)^{B(K,G)-1},
\]
where $B(K,G)$ is the corrected power of $\log X$ given by Turkelli \cite{turkelli2015}. This satisfies Turkelli's correction to Malle's predicted strong lower bound, which is always greater than or equal to Malle's original predicted strong lower bound.
}
\end{itemize}
\end{corollary}

These lower bounds can be considered the greatest possible generalization of Kl\"uners' arguments showing that $C_3\wr C_2$ is a counter-example to Malle's conjecture. This is a great improvement on known lower bounds, realizing conjecturally sharp bounds in many cases. As a consequence, we prove nontrivial lower bounds for every solvable group over every base field by noting that the inverse Galois problem is true for solvable groups and that every solvable group has a nontrivial abelian normal subgroup (for example, the Socle).

\begin{corollary}\label{cor:solvable}
For any solvable transitive subgroup $G\subset S_n$ and any number field $K$, there exists an integer $0<a<n$ depending only on $G$ such that
\[
N(K,G;X) \gg X^{\frac{1}{a}}\,.
\]
In particular, Corollary \ref{cor:lowerbound} implies that we can choose
\[
a = \min\{\ind(g) : g\in G-\{1\} \text{ and }g\text{ commutes with its conjugates}\}\,.
\]
\end{corollary}


For many solvable groups this is the first known nontrivial lower bound, and is at least as large as $X^{\frac{1}{n-1}}$. These bounds are at least as good as the bounds for groups with a central subgroup proven by Kl\"uners-Malle \cite{kluners-malle2004}, and strictly better than bounds for solvable groups with regular polynomials proven by Pierce-Turnage-Butterbaugh-Wood \cite{pierce-turnage-butterbaugh-wood2017}.

\subsection{Layout of the paper}

This paper is made up of four sections.

Section \ref{sec:analytic} covers the analytic results we require for this paper. This involves locating poles of Euler products and using Tauberian theorems to convert the analytic information at a pole of a Dirichlet series to asymptotic information of the corresponding arithmetic function. These kinds of analytic number theory tools are standard in the literature on arithmetic statistics, however we will still need to prove  that such tools work in the generality that we require.

In Section \ref{sec:towers}, we give a more detailed discussion of the towers $F/L/K$ and prove Lemma \ref{lem:Z1} stated above. We prove that the Malle-Bhargava principle gives the prediction listed in the introduction by computing the rightmost pole of an appropriate Euler product of local factors. This highlights the analogy with Malle's original conjecture and provides compelling evidence that counting $1$-coclasses is a natural generalization with similar behavior. We also discuss two other important considerations:
\begin{itemize}

\item{We show that the $\pi$-discriminant factors through the coboundary relation, implying that all of the statements in the introduction apply equally well to the first cohomology group $H^1(K,T(\pi))$. When $T$ is abelian, we show that the Malle-Bhargava principle has an equivalent form expressed in terms of local cohomology groups $H^1(K_p,T(\pi))$. These results will be important for the proofs of the main results in Sections \ref{sec:AsymptoticWiles} and \ref{sec:counting}, and will allow us to make use of powerful local-to-global tools in Galois cohomology.}

\item{We will discuss issues of consistency in the Malle-Bhargava principle, and the relationship to Malle's original conjecture. Of particular interest is the relationship of this refined problem to counter-examples to Malle's conjecture. We will specifically address Kl\"uners' counter-example $G=C_3\wr C_2$, and show that this follows from an overlap in the composite action $\pi*\chi^{-1}$ for certain $\pi:G_K\rightarrow C_2$. Malle's original conjecture essentially assumes independence of the action by conjugation $\pi$ and the cyclotomic action, which Kl\"uners' counter-example shows is just not always true. We make a comparison of this insight with Turkelli's proposed correction to Malle's conjecture \cite{turkelli2015}, showing that Turkelli's corrections predict this behavior and suggests that this composite action $\pi*\chi^{-1}$ is the more natural relation to consider when counting towers and 1-coclasses.}
\end{itemize}

We will prove a more general result about counting elements of $H^1(K,T(\pi))$ with bounded discriminants in Section \ref{sec:AsymptoticWiles}, from which Theorem \ref{thm:MalleZ1} will be a special case. The nontrivial Galois action on $T$ prevents us from following the same approach Wright uses to count abelian extensions, as 1-coclasses will not always factor through the group of ideles. Rather than approaching the problem via idelic class field theory, we take a different approach via a theorem of Wiles \cite{wiles1995} on generalized Selmer groups. If $\mathcal{L}=(L_p)$ is a family of subgroups $L_p\le H^1(K_p,T)$ of local cohomology groups, Wiles defines the corresponding generalized Selmer group to be
\[
H^1_{\mathcal{L}}(K,T) = \{f\in H^1(K,T) : \forall p,\res_p(f)\in L_p\}.
\]
If $L_p=H^1_{ur}(K_p,T)$ is the kernel of the restriction to $H^1(I_p,T)$ for all but finitely many places $p$, Wiles proves that
\[
|H^1_{\mathcal{L}}(K,T)| \approx \prod_p \frac{|L_p|}{|H^0(K_p,T)|},
\]
which is approximately a product of local densities. We use this to decompose the Dirchlet series
\[
\sum_{f\in H^1_{\mathcal{L}}(K,T)} \mathcal{N}_{K/\Q}(\disc(f))^{-s}
\]
into a finite sum of Euler products, from which we explicitly compute a meromorphic continuation and the rightmost poles. Applying a general Tauberian theorem will prove a so-called ``aymptotic Wiles' Theorem" for counting 1-coclasses in an infinite Selmer group $H^1_{\mathcal{L}}(K,T)$ with bounded discriminant. We compare this new approach to a modification of the classical methods used by Wright to prove Malle's original conjecture for abelian groups in Appendix \ref{app}.

We conclude with Section \ref{sec:counting} on number field counting, where we give the explicit proofs of Theorem \ref{thm:MalleZ1}, Corollary \ref{cor:alpha}, and Corollary \ref{cor:ram} as special cases of the asymptotic Wiles' Theorem. From here we prove the lower bounds for Malle's conjecture given in Corollary \ref{cor:lowerbound} and Corollary \ref{cor:solvable}.

\section*{Acknowledgements}
I would like to thank Melanie Matchett Wood for a multitude of feedback and discussions on various drafts of this paper, as well as Jiuya Wang with whom I had many helpful conversations towards fitting this project into the bigger picture of Malle's conjecture. I would also like to thank Adebisi Agboola, Harsh Mehta, Evan O'Dorney, Ila Varma, and the anonymous referees for helpful comments.

\newpage

\section{Analytic Preliminaries}\label{sec:analytic}

The primary analytic tools we will use to convert algebraic information into asymptotic information will be Dirichlet series and Tauberian Theorems. We will primarily be concerned with Dirichlet series with an Euler product whose Euler factors are ``Frobenian'' in the sense of \cite{serre2012}:

\begin{definition}
Let $\Omega$ be a set. We call a function $\varphi:\{\text{places of }K\} \rightarrow \Omega$ \textbf{Frobenian in }$\mathbf{F/K}$ if there exists a finite set of places $S$ and a class function $\Gal(F/K)\rightarrow \Omega$ (also denoted $\varphi$ by abuse of notation) such that for any $p\not\in S$,
\[
\varphi(p) = \varphi(\Fr_p)\,.
\]
\end{definition}

Frobenian functions occur implicitly in the original Malle-Bhargava principle, as the local Euler factors depend on the class $\Leg{\Q(\mu_n)/\Q}{p}$, where we use $\Leg{F/K}{\cdot}$ to denote the Artin map. In the setting of this paper, the Galois action induced by $\pi$ on $T$ will specify which field the Euler factors are Frobenian with respect to. As we can choose $\pi$ to be arbitrary, it will be convenient to treat Frobenian functions in generality.

Frei-Loughran-Newton \cite{frei-loughran-newton2018,frei-loughran-newton2019} utilize these ideas to count abelian extensions with infinitely many local conditions, and in particular \cite[Proposition 2.3]{frei-loughran-newton2019} shows that if $\rho$ is a Frobenian function outside of $S$ then the series
\[
\prod_{p\not\in S} \left(1 + \rho(p) \mathcal{N}_{K/\Q}(p)^{-s}\right)
\]
factors as
\[
\zeta_K(s)^{m(\rho)} G(s)\,,
\]
where $G(s)$ is holomorphic and zero-free on some open neighborhood of ${\rm Re}(s)\ge 1$ and $m(\rho)$ is the mean of the Frobenian function
\[
m(\rho)= \frac{1}{[F:K]}\sum_{\sigma\in \Gal(F/K)} \rho(\sigma)\,.
\]
A special case of this result is also found in work of Kaplan-Marcinek-Takloo-Bighash \cite{kaplan-marcinek-takloo-bighash2015}.

We need a (slightly) more general result on Euler products associated to Frobenian functions. Frei-Loughran-Newton's result works well for counting abelian fields ordered by conductor, specifically because there are only two options for $\nu_p(\text{conductor})$ at all but finitely many places: $0$ or $1$ depending on if $p$ is ramified or not. The discriminant and $\pi$-discriminant both allow for more general powers of $p$, so we prove an extension of \cite[Proposition 2.3]{frei-loughran-newton2019}:

\begin{proposition}\label{prop:GeneralDirichlet}
Suppose $Q_p(x)\in \C[x]$ is Frobenian in $F/K$ such that for each $\sigma\in \Gal(F/K)$
\[
Q_\sigma(x)\in 1 + x\C[x]\,.
\]
Then there exist constants $a(Q)$ and $b(Q)$ such that
\[
\prod_{p} Q_p(\mathcal{N}_{K/\Q}(p)^{-s}) = \zeta_K(a(Q)s)^{b(Q)} G(s)\,,
\]
where $G(s)$ is holomorphic on some open neighborhood of ${\rm Re}(s)\ge 1/a(Q)$, which are given by
\begin{align*}
a(Q) &= \min_{\sigma\in \Gal(F/K)} -\deg\left(Q_\sigma(1/x) - 1\right)\\
&= \min_{\sigma\in \Gal(F/K)} \textnormal{smallest nonzero power of }x\textnormal{ in } Q_\sigma(x)\\
b(Q) &= \frac{1}{[F:K]} \sum_{\sigma\in \Gal(F/K)} \lim_{x\rightarrow 0} \frac{Q_\sigma(x)-1}{x^{a(Q)}}\\
& = \textnormal{mean value of the coefficient of }x^{a(Q)}\textnormal{ in }Q_\sigma(x)\,.
\end{align*}
Moreover, $G(s)=0$ for ${\rm Re}(s)\ge 1/a(Q)$ if and only if $Q_p(\mathcal{N}_{K/\Q}(p)^{-s})=0$ for some place $p$.
\end{proposition}

\begin{proof}[Proof of Proposition \ref{prop:GeneralDirichlet}]
For convenience, we expand the polynomials functions $Q_p(x)$ as
\[
Q_p(x) = \sum_{i=0}^{N} q(p,i) x^{i}
\]
for $N$ some large finite number independent of $p$ (since $Q_p$ is Frobenian, there are only finitely many possible polynomials for $Q_p(x)$ thus we can choose $N$ to be the maximum of their degrees). Let $S$ be the finite set of places that don't agree with the class function. The constants $a(Q)$ and $b(Q)$ can be written in terms of the coefficients as
\begin{align*}
a(Q) &= \min_{\sigma\in \Gal(F/K)} \min_{\substack{i\ne 0\\ q(\sigma,i)\ne 0}} i\\
b(Q) &= \frac{1}{[F:K]}\sum_{\sigma\in \Gal(F/K)} q(\sigma,a(Q))\,.
\end{align*}
We define the function
\begin{align*}
G_1(s) &= \prod_{p\in S} Q_p(\mathcal{N}_{K/\Q}(p)^{-s}) \prod_{p\not\in S} \frac{Q_p(\mathcal{N}_{K/\Q}(p)^{-s})}{1 + q(p,a(Q)) \mathcal{N}_{K/\Q}(p)^{-a(Q)s}}\,.
\end{align*}
$Q_p(x)$ being Frobenian in $F/K$ implies the map $p\mapsto q(p,a(Q))$ is also Frobenian in $F/K$ so that \cite[Proposition 2.3]{frei-loughran-newton2019} implies
\[
\prod_{p\not\in S} \left(1 + q(p,a(Q)) \mathcal{N}_{K/\Q}(p)^{-a(Q)s}\right) = \zeta_K(a(Q)s)^{b(Q)} G_2(a(Q)s)\,,
\]
where $G_2(s)$ is holomorphic and zero-free on an open neighborhood of ${\rm Re}(s)\ge 1$. This implies
\[
\prod_{p} Q_p(\mathcal{N}_{K/\Q}(p)^{-s}) =  \zeta_K(a(Q)s)^{b(Q)} G_2(a(Q)s)G_1(s)\,.
\]
Therefore it suffices to show that $G_1(s)$ is holomorphic in an open neighborhood of ${\rm Re}(s)\ge 1/a(Q)$.

The product over $p\in S$ is a finite product of sums of powers of $\mathcal{N}_{K/\Q}(p)^{-s}$, and so is necessarily holomorphic. Set $x=\mathcal{N}_{K/\Q}(p)^{-s}$, then each local factor $p\not\in S$ satisfies
\begin{align*}
\Big|\frac{Q_p(x)}{1 + q(p,a(Q)) x^{a(Q)}}\Big| &= \Bigg|\frac{\sum_{i=0}^{N} q(p,i) x^i}{1+q(p,a(Q)) x^{a(Q)}}\Bigg|\,.
\end{align*}
$Q_p(x)$ is Frobenian, which implies that there exists some $\sigma\in \Gal(F/K)$ such that $Q_p(x)=Q_\sigma(x)$. Therefore $q(p,i)=0$ for all $i>\deg Q_\sigma$. Moreover, the definition of $a(Q)$ implies $q(p,i)=0$ for all $0<i<a(Q)$. Lastly, $q(p,0)=1$. We can compute the first several terms of the summation to find that
\begin{align*}
\Big|\frac{Q_p(x)}{1 + q(p,a(Q)) x^{a(Q)}}\Big|&\le 1 + \Bigg|\frac{\sum_{i=a(Q)+1}^{N} q(\sigma,i) x^{i}}{1+q(\sigma,a(Q)) x^{a(Q)}}\Bigg|\\
&\le 1 + \frac{\sum_{i=a(Q)+1}^{N} |q(\sigma,i)| \cdot |x|^{i-a(Q)-1}}{1 - |q(\sigma,a(Q))|\cdot |x|^{a(Q)}} \cdot |x|^{a(Q)+1}\,.
\end{align*}
In particular, if we set
\[
C = \frac{1}{2}\max_{\sigma\in \Gal(F/K)} \max_{i} |q(\sigma,i)|\,,
\]
it follows that for $x$ satisfying
\[
|x|< \min\left\{\left(\frac{1}{2|q(\sigma,a(Q))|}\right)^{1/a(Q)}, 1\right\}
\]
this produces an upper bound
\begin{align*}
\Big|\frac{Q_p(x)}{1 + q(p,a(Q)) x^{a(Q)}}\Big| &\le 1+ C\cdot |x|^{a(Q)+1}\,.
\end{align*}
Choose a finite set of places $S$ sufficiently large so that
\[
|\mathcal{N}_{K/\Q}(p)|\ge \min_{\sigma}\left(\frac{1}{2}|q(\sigma,a(Q))|\right)
\]
implies $p\in S$. Taking a product over these bounds implies that for any ${\rm Re}(s)>1/a(Q)$
\begin{align*}
\Big|\prod_{p\not\in S} \frac{Q_p(\mathcal{N}_{K/\Q}(p)^{-s})}{1 + q(p,a(Q)) \mathcal{N}_{K/\Q}(p)^{-a(Q)s}}\Big| &\le \prod_{p\not\in S} (1 + C\mathcal{N}_{K/\Q}(p)^{-(a(Q)+1){\rm Re}(s)})\\
&\le \sum_{I} C^{\#\{p\mid I\}}\mathcal{N}_{K/\Q}(I)^{-(a(Q)+1){\rm Re}(s)}\\
&\le \sum_{I} \mathcal{N}_{K/\Q}(I)^{-(a(Q)+1){\rm Re}(s)+\epsilon}\\
&=\zeta_K((a(Q)+1){\rm Re}(s)-\epsilon)\,,
\end{align*}
which converges absolutely on the region ${\rm Re}(s)>\frac{1+\epsilon}{a(Q)+1}$ for each choice of $\epsilon>0$. This contains the region ${\rm Re}(s)\ge 1/a(Q)$, which implies $G_1(s)$ converges absolutely on this region, and so is in particular holomorphic. The zeros of an absolutely convergent Euler product are exactly the zeros of its factors, which implies the zeroes of $G_1(s)$ on this region are exactly the zeros of $Q_p(\mathcal{N}_{K/\Q}(p)^{-s})$ for some $p$.
\end{proof}

This is the appropriate setup for a Tauberian theorem. There are a multitude of such theorems to choose from, and we will make use of Delange's Tauberian theorem \cite[Theorem III]{delange1954}:

\begin{theorem}[Theorem III \cite{delange1954}]\label{thm:delange}
Let $F(s)=\sum_{n=1}^{\infty} f(n)n^{-s}$ be a Dirichlet series. Suppose there exists a complex number $a$, a real number $w$ such that $w\not\in \Z_{\le 0}$, and functions $h(s)$ and $g(s)$ which are holomorphic for ${\rm Re}(s)\ge {\rm Re}(a)$ for which
\[
F(s) = (s-a)^{-w} g(s) + h(s)\,.
\]
Then
\[
\sum_{n\le X} f(n) \sim \frac{g(a)}{\Gamma(w)}X^a (\log X)^{w-1}
\]
as $X\to \infty$.
\end{theorem}

If we are in the setting of Proposition \ref{prop:GeneralDirichlet}, we get the following result:

\begin{corollary}\label{cor:maintauberian}
Let $F(s) = \sum_{I} f(I) \mathcal{N}_{K/\Q}(I)^{-s}$ be a Dirichlet series with an Euler product
\[
F(s) = \prod_{p} Q_p(\mathcal{N}_{K/Q}(p)^{-s})\,.
\]
If $\{Q_p\}$ satisfy the hypotheses of Proposition \ref{prop:GeneralDirichlet} with $b(Q)\in\R - \Z_{\le 0}$ and $Q_p(\mathcal{N}_{K/\Q}(p)^{-1/a(Q)})\ne 0$ for each place $p$, then
\[
\sum_{\mathcal{N}_{K/\Q}(I)\le X} f(I) \sim \frac{G(1)}{a(Q)^{b(Q)}\Gamma(b(Q))}\left(\Res_{s=1}\zeta_K(s)\right)^{b(Q)} X^{1/a(Q)}(\log X)^{b(Q)-1}
\]
as $X\to \infty$.

If instead $b(Q)\in\Z_{\le 0}$, then for each $\epsilon>0$
\[
\sum_{\mathcal{N}_{K/\Q}(I)\le X} f(I) = O\left(X^{1/a(Q)}(\log X)^{-1+\epsilon}\right)
\]
as $X\to \infty$.
\end{corollary}

\begin{proof}
If $b(Q)\in\R-\Z_{\le 0}$, then we can write
\begin{align*}
\zeta_K(a(Q)s)^{b(Q)}G(s) &= (a(Q)s-1)^{-b(Q)}\left[(a(Q)s-1)\zeta_K(a(Q)s)\right]^{b(Q)}G(s)\\
&= (s-1/a(Q))^{-b(Q)} a(Q)^{-b(Q)}\left[(a(Q)s-1)\zeta_K(a(Q)s)\right]^{b(Q)} G(s)\,.
\end{align*}
The Dedekind zeta function has a single pole at $s=1$ of order $1$. We then set
\[
g(s) = a(Q)^{-b(Q)}\left[(a(Q)s-1)\zeta_K(a(Q)s)\right]^{b(Q)} G(s)\,,
\]
which is holomorphic for ${\rm Re}(s)\ge 1/a(Q)$ and satisfies
\[
g(1) = \frac{G(1)}{a(Q)^{b(Q)}}\left(\Res_{s=1}\zeta_K(s)\right)^{b(Q)}\,.
\]
Applying Theorem \ref{thm:delange} concludes the proof of the first case.

The second case with $b(Q)\in \Z_{\le 0}$ is not directly addressed by Delange, but we remark that if $b(Q)$ is a negative integer then the pole at $s=1/a(Q)$ of order $b(Q)$ is really a zero of order $-b(Q)$ and $F(s)$ is holomorphic on ${\rm Re}(s)\ge 1/a(Q)$. We write
\[
\zeta_K(a(Q)s)^\epsilon = \sum_I z_\epsilon^{a(Q)}(I) \mathcal{N}_{K/\Q}(I)^{-s}\,.
\]
Theorem \ref{thm:delange} does apply to this function with $w=\epsilon$, and implies
\[
\sum_{\mathcal{N}_{K/\Q}(I)\le X} z_\epsilon^{a(Q)}(I) \sim \frac{\left(\Res_{s=1}\zeta_K(s)\right)^{\epsilon}}{a(Q)^{\epsilon}\Gamma(\epsilon)}X^{1/a(Q)}(\log X)^{-1+\epsilon}
\]
as $X\to \infty$. But also,
\[
F(s)+\zeta_K(a(Q)s)^{\epsilon} = (s-1/a(Q))^{-\epsilon}\left[a(Q)^{-\epsilon}(a(Q)s-1)^{\epsilon}\zeta_K(a(Q)s)^{\epsilon}\right] + F(s)
\]
satisfies the hypotheses of Theorem \ref{thm:delange}, which implies
\[
\sum_{\mathcal{N}_{K/\Q}(I)\le X} f(I) + z_\epsilon^{a(Q)}(I) \sim \frac{\left(\Res_{s=1}\zeta_K(s)\right)^{\epsilon}}{a(Q)^{\epsilon}\Gamma(\epsilon)}X^{1/a(Q)}(\log X)^{-1+\epsilon}\,.
\]
By subtracting the two results, we find that
\[
\sum_{\mathcal{N}_{K/\Q}(I)\le X} f(n) = o\left(X^{1/a(Q)}(\log X)^{-1+\epsilon}\right)\,.
\]
\end{proof}

\newpage

\section{Counting Towers of Number Fields}\label{sec:towers}

Fix a transitive subgroup $G\subset S_n$. We recall the $\frac{|G|}{n}$-to-$1$ correspondence discussed in the introduction:
\begin{align*}
\left\{\substack{\ds L/K \text{ degree }n\text{ with }\Gal(\widetilde{L}/K)\cong G\\\ds\text{ordered by }\disc(L/K)}\right\} \leftrightarrow\left\{\substack{\ds \widetilde{L}/K \text{ Galois with }\Gal(\widetilde{L}/K)\cong G\\\ds\text{ordered by the discriminant}\\\ds\text{of a subfield fixed}\\\ds\text{by a point stabilizer}}\right\}\,.
\end{align*}
We will work entirely on the right-hand side of this correspondence, and all field extensions will be understood to be Galois unless stated otherwise.

\subsection{Preliminaries}

Throughout this section, fix a transitive subgroup $G\subset S_n$ and a normal subgroup $T\normal G$. We count extensions with multiplicity, i.e. we count pairs $(F/L/K, \gamma)$ for which $\gamma$ is an isomorphism between $\Gal(F/K) \rightarrow G$.
\begin{definition}
If $(L/K,\iota_B)$ is an extension together with an isomorphism $\iota_B:\Gal(L/K) \xrightarrow{\sim} B$, we call $(F/L/K,\gamma)$ a \textbf{$\mathbf{(T\normal G)}$-tower} if $\gamma$ is an isomorphism $\Gal(F/K) \xrightarrow{\sim} G$ and $\gamma \equiv \iota_B \mod T$. Define the counting function
\[
N(L/K,T\normal G;X) = \#\{(F/L/K,\gamma)\ (T\normal G)\text{-tower}: \mathcal{N}_{K/\Q}(\disc(F^H/K))<X\}\,,
\]
where $H={\rm Stab}_G(1)$ is a point stabilizer in $G$.
\end{definition}

The Galois correspondence gives a bijection between $(T\normal G)$-towers and surjective homomorphisms $\gamma:G_K\twoheadrightarrow G$ which are equal to $\iota_B$ after composition with the quotient map $G\rightarrow B$, written $\gamma\equiv \iota_B \mod T$. Lemma \ref{lem:Z1} gives an alternate formulation of such homomorphisms via a bijection with the set of crossed homomorphisms, or $1$-cocycles, whenever there exists at least one $(T\normal G)$-tower given by $\pi$. This gives a bijection
\[
\left\{(F/L/K,\gamma)\ (T\normal G)\text{-tower}\right\} \leftrightarrow \left\{f\in Z^1(G_K,T(\pi)) \mid f*\pi\text{ surjective}\right\}\,.
\]

\begin{proof}[Proof of Lemma \ref{lem:Z1}]
Consider the quotient map
\[
\begin{tikzcd}
\Hom(G_K,G)\rar{q_*} &\Hom(G_K,G/T),
\end{tikzcd}
\]
and fix some $\pi\in \Hom(G_K,G)$. Then
\[
q_*^{-1}(q_*(\pi)) = \{\pi'\in \Hom(G_K,G) : \pi(x)T=\pi'(x)T,\ x\in G_K\}.
\]
Suppose $q_*(\pi')=q_*(\pi)$. Then $\pi'*\pi^{-1}$ is a map $G_K\rightarrow T$, and
\begin{align*}
(\pi'*\pi^{-1})(xy) &= \pi'(x)\pi'(y)\pi^{-1}(y)\pi^{-1}(x)\\
&=(\pi'*\pi^{-1})(x)\cdot c_{\pi(x)}((\pi'*\pi^{-1})(y))
\end{align*}
is a crossed homomorphism in $Z^1(G_K,T(\pi))$ and $c_g(t) = gtg^{-1}$. Conversely, if $f\in Z^1(G_K,T(\pi))$, then $f*\pi:G_K\rightarrow G$ is a homomorphism as
\begin{align*}
(f*\pi)(xy) &= f(xy)\pi(xy)\\
&=f(x)c_{\pi(x)}(f(y))\pi(x)\pi(y)\\
&=f(x)\pi(x)f(y)\pi(y)\\
&=(f*\pi)(x)(f*\pi)(y).
\end{align*}
\end{proof}

We introduced the $\pi$-discriminant on crossed homomorphisms to be the pull-back of the discriminant on towers via the isomorphism described in Lemma \ref{lem:Z1}
\[
\disc_\pi(f) = \disc(f*\pi)\,.
\]
We need to be clear about what we mean by $\disc(f*\pi)$. If $f*\pi$ is surjective, then it corresponds to a $(T\normal G)$-tower $F/L/K$ with $F$ being the fixed field of $f*\pi$. We defined the counting function
\[
N(L/K,T\normal G;X) = \#\left\{(F/L/K,\gamma) : \mathcal{N}_{K/\Q}(\disc(F^H/K))<X\right\}\,,
\]
so we take $\disc(f*\pi)$ to be the usual discriminant of the subfied $F^H$ of $F$ fixed by a point stabilizer whenever $f*\pi$ is surjective. We want to extend this definition to non-surjective homomorphisms, so that we can instead compute the size of the sets
\[
Z^1(K,T(\pi);X) = \left\{ f\in Z^1(K,T(\pi)) : \disc_\pi(f)<X\right\}
\]
and then perform a M\"obius inversion to obtain information on $N(L/K,T\normal G;X)$. For non-surjective homomorphisms, we no longer want to take the usual discriminant of a subfield of the fixed field. The essential property we need our discriminant to have is that it is determined locally, i.e. $\nu_p(\disc(\pi))$ depends only on $\pi|_{I_p}$. The degree of the fixed field is a global property, and if that degree changes it can change the discriminant of the fixed field.

\textbf{Example:}
Fix an ismorphism $G_\Q^{ab} = \prod_{p<\infty} I_p(\Q^{ab}/\Q)$ via Kronecker-Weber and let $\tau_p$ be a generator of tame inertia at $p$. We define two tamely ramified homomorphisms $\pi_1,\pi_2:G_\Q\rightarrow \Z/4\Z$ by
\begin{align*}
\pi_1(\tau_p) &= \begin{cases}
2 & p=3\\
1 & p=5\\
0 & p\nmid 15\infty\,,
\end{cases}
&&\pi_2(\tau_p) = \begin{cases}
2 & p=3\\
0 & p\nmid 3\infty\,.
\end{cases}
\end{align*}
The first map is surjective and tamely ramified. $\Z/4\Z\subset S_4$ necessarily has the regular representation, which implies the power of a prime dividing the discriminant is given by $4 - \#\{\text{orbits }\pi_1(\tau_p)\}$ so that
\[
|\disc(\pi_1)| = 3^{4 - 2} \cdot 5^{4 - 1} = 3^2\cdot 5^3\,.
\]
Meanwhile, the fixed field of $\pi_2$ is a quadratic field ramified only at $3$, i.e. is equal to $\Q(\sqrt{-3})$. This discriminant is given by
\[
|\disc(\Q(\sqrt{-3})/\Q)|=3\,.
\]
However, $\pi_1$ and $\pi_2$ are locally the same at the place $3$. This shows that the degree of the fixed field is some global invariant affecting the discriminant of the fixed field.

We instead define the discriminant via the Galois correspondence to \'etale algebras:

\begin{definition}
Let $G\subset S_n$ be a transitive subgroup. Then there is a (many-to-1) Galois correspondence between $\Hom(G_K,S_n)$ and dimension $n$ \'etale algebras $F/K$ (see Chapter V, Section 6, Proposition 12 of \cite{bourbaki2003}). For any $\pi\in \Hom(G_K,G) \subset \Hom(G_K,S_n)$, we take $\disc(\pi)$ to mean the discriminant of the \'etale algebra corresponding to $\pi$.

If $\pi$ is surjective, this agrees with the usual discriminant of the degree $n$ extension corresponding to $\pi$.
\end{definition}

The discriminant on \'etale algebras is defined locally, and is the appropriate choice for us to define $Z^1(K,T(\pi);X)$. We remark that under this definition, if $p$ is tamely ramified then
\[
\nu_p(\disc(\pi)) = n - \#\{\text{orbits of }\pi(I_p)\}\,,
\]
which agrees with $\ind(g)$ where $g\in G$ is a generator of $\pi(I_p)$ (see for example \cite{koch2000}).

\subsection{The Malle-Bhargava Principle}

We will describe the asymptotic size of $N(L/K,T\normal G;X)$ by considering the analytic behavior of
\[
\sum_{\substack{(F/L/K,\gamma)\\(T\normal G)\text{-towers}}} \mathcal{N}_{K/\Q}(\disc(F^H/K))^{-s}\,.
\]
This Dirichlet series is equivalent to the following series as a consequence of Lemma \ref{lem:Z1}
\[
\sum_{\substack{f\in Z^1(K,T(\pi))\\ \text{surjective}}} \mathcal{N}_{K/\Q}(\disc_\pi(f))^{-s}\,.
\]
Crossed homomorphisms behave very similarly to homomorphisms. In particular, every place $p$ of $K$ comes with a restriction map
\begin{align*}
\res_p:Z^1(K,T(\pi)) \rightarrow Z^1(K_p,T(\pi))\,.
\end{align*}
The Malle-Bhargava principle suggests that this series should behave like an Euler product of local factors. Noting that $Z^1(K,T(\pi))$ specializes to $\Hom(G_K,T)$ under the trivial action, we can consider the Euler product
\[
\prod_p \frac{1}{|T|} \left(\sum_{f_p\in Z^1(K_p,T(\pi))} \mathcal{N}_{K/\Q}(\disc_\pi(f_p))^{-s}\right)
\]
as a natural generalization of the local series given by Bhargava when $T$ has the trivial action
\[
\prod_p \frac{1}{|T|} \left(\sum_{f_p\in \Hom(K_p,T)} \mathcal{N}_{K/\Q}(\disc(f_p))^{-s}\right)\,.
\]
The Malle-Bhargava principle states that this local series should be arithmetically equivalent to the global series, i.e. it should have the same rightmost pole of the same order.

\begin{theorem}\label{thm:MBlocal}
Let $T$ be a group with a Galois action $\pi:G_K\rightarrow \Aut(T)$. Then the Dirichlet series
\[
\prod_p \frac{1}{|T|}\left(\sum_{f_p\in Z^1(K_p,T(\pi))} \mathcal{N}_{K/\Q}(\disc_\pi(f_p))^{-s}\right)
\]
has a meromorphic continuation to an open neighborhood of ${\rm Re}(s)\ge 1/a(T)$ with a single pole at $s=1/a(T)$ of order $b(K,T(\pi))$, where
\begin{align*}
a(T) &= \min_{t\in T-\{1\}} \ind(t)\\
b(K,T(\pi)) &= \#\left(\{\text{conjugacy class }C\subset A(T)\}/\pi*\chi^{-1}\right)\,,
\end{align*}
with $A(T)=\{t\in T\mid \ind(t)=a(T)\}$ and $\chi:G_K\rightarrow \hat{\Z}^\times$ the cyclotomic character.
\end{theorem}

This generalizes the behavior of the local series proposed by the original Malle-Bhargava principle, and applying a Tauberian theorem (such as Delange's Theorem \ref{thm:delange}) gives the prediction outlined in the introduction
\[
N(L/K,T\normal G;X) \sim c'(K,T(\pi)) X^{1/a(T)} (\log X)^{b(K,T(\pi))-1}\,.
\]

We prove this by making use of the analytic tools in Section \ref{sec:analytic}. We first prove that the maps $p\mapsto Z^1(K_p,T(\pi))$ and $p\mapsto H^1(K_p,T(\pi))$ are Frobenian in $F/K$ for a particularly nice choice of $F$. The following lemma does this explicitly by constructing natural isomorphisms with cohomology groups depending only on $\Leg{F/K}{p}$.

\begin{lemma}\label{lem:KC}
Let $F/K$ be any finite extension containing the field of definition of $T(\pi)$ and the roots of unity $\mu_{|T|}$. Let $p$ be any place of $K$ such that $p\nmid |T|\infty$, $p$ unramified in $F$, and $\mathcal{N}_{K/\Q}(p)\equiv m\mod |T|$. Define
\[
G_m = \langle \tau, \Fr : c_{\Fr}(\tau)=\tau^{m}\rangle.
\]
Then the following hold:
\begin{itemize}
\item[(i)]{
Fix an embedding $G_{K_p}\hookrightarrow G_K$, inducing a Galois action of $G_{K_p}$ on $T$. Then $G_m/\langle \tau^{|T|}\rangle$ is isomorphic to a dense subgroup of $G_{K_p}^{\rm tame}/I_p^{|T|}$ with $\tau$ sent to a generator of inertia and $\Fr$ sent to Frobenius, and moreover the inflation map induces an isomorphism $Z^i(G_m,T(\pi))\xrightarrow{\sim} Z^i(G_{K_p},T(\pi))$ for each $i=0,1$.
}

\item[(ii)]{
The inflation isomorphism in part (i) is natural with respect to the choice of embedding, i.e. if $g\in G_K$ and $x\mapsto c_g(x)$ is another embedding $G_{K_p}\hookrightarrow G_K$, then conjugation by $g$ induces an isomorphism on crossed homomorphisms and the following diagram of isomorphisms commutes:
\[
\begin{tikzcd}
Z^i(G_m,T(\pi))\dar{\infl} \rar{c_g} & Z^i(G_m,T(\pi))\dar{\infl}\\
Z^i(G_{K_p},T(\pi)) \rar{c_g} &Z^i(G_{K_p},T(\pi))
\end{tikzcd}
\]
for $i=0,1$.
}

\item[(iii)]{
The isomorphisms $Z^i(G_m,T(\pi))\xrightarrow{\sim} Z^i(G_{K_p},T(\pi))$ factor through the coboundary relation and induce isomorphisms $H^i(G_m,T(\pi))\xrightarrow{\sim} H^i(G_{K_p},T(\pi))$.
}
\end{itemize}
We denote
\begin{align*}
Z^i(K_\sigma,T(\pi)) &:= Z^i(G_m,T(\pi))\,, & H^i(K_\sigma,T(\pi)) &:= H^i(G_m,T(\pi))\,,
\end{align*}
which are determined by $\Leg{F/K}{p}=\sigma$ uniquely up to conjugation on the action. We also denote
\begin{align*}
Z^1_{ur}(K_\sigma,T(\pi)) &:= Z^1(\langle \Fr\rangle, (T(\pi))^{\langle \tau\rangle})\,, & H^1_{ur}(K_\sigma,T(\pi)) &:= H^1(\langle \Fr\rangle, (T(\pi))^{\langle \tau\rangle})
\end{align*}
to be the kernels of the restriction to $\langle \tau\rangle$ map.
\end{lemma}

\begin{proof}
Fix a place $p$ satisfying the required hypotheses. Because $p$ is not ramified in $F$, it follows that $I_p$ acts trivially on $T$ and
\[
Z^i(I_p,T(\pi)) = \begin{cases}
T & i=0\\
\Hom(I_p,T) & i=1.
\end{cases}
\]
If $i=0$, this implies $Z^0(G_{K_p},T)=T=Z^0(G_m,T)$, so this case of part (i) is trivial.

If $i=1$, we note that $p\nmid |T|$ implies that $\Hom(I_p,T)$ factors through tame inertia and through the quotient $I_p/I_p^{|T|}$. This implies inflation $Z^1(G_{K_p}^{\rm tame}/I_p^{|T|},T(\pi))\rightarrow Z^1(K_p,T(\pi))$ is an isomorphism. The explicit presentation of $G_{K_p}^{\rm tame}$ as a profinite group
\[
G_{K_p}^{\rm tame} = \langle \tau_p, \Fr_p : c_{\Fr_p}(\tau_p) = \tau_p^{\mathcal{N}_{K/\Q}(p)}\rangle
\]
is exactly the same as for $G_m$ as a discrete group, where we note
\[
\mathcal{N}_{K/\Q}(p)\equiv \Leg{K(\mu_{|T|})/K}{p}\mod |T|
\]
follows from class field theory, and
\[
\Leg{F/K}{p}\equiv \Leg{K(\mu_{|T|})/K}{p}\mod |T|
\]
follows from $K(\mu_{|T|})\subset F$. This implies $G_m/\langle \tau^{|T|}\rangle$ embeds in $G_{K_p}^{\rm tame}/I_p^{|T|}$ naturally by $\tau\mapsto \tau_p$ and $\Fr\mapsto \Fr_p$ with a dense image. $T$ finite and all coclasses continuous implies that the restriction map induces an isomorphism $Z^1(G_{K_p}^{\rm tame}/I_p^{|T|},T)\cong Z^1(G_m,T)$. This concludes the proof of part (i).

Part (ii) follows from the naturality of restriction and inflation maps, which are the maps used to induce the isomorphisms
\[
Z^i(G_m,T)\xleftarrow{\res} Z^i(G_{K_p}^{\rm tame}/I_p^{|T|},T)\xrightarrow{\inf} Z^i(G_K,T).
\]

Part (iii) then follows from the fact that the restriction and inflation maps are known to factor through the coboundary relation.
\end{proof}

This implies $Z^1(K_p,T(\pi))$ is Frobenian, and is all the set up we need to prove Theorem \ref{thm:MBlocal}.

\begin{proof}[Proof of Theorem \ref{thm:MBlocal}]
We consider the Euler factors as coming from polynomials
\[
Q_p(x) = \frac{1}{|T|}\sum_{f_p\in Z^1(K_p,T(\pi))} x^{\nu_p(\disc_\pi(f_p))}\,.
\]
We will prove that these polynomials satisfy the hypotheses of Proposition \ref{prop:GeneralDirichlet}, i.e. they are Frobenian in a field extension $F/K$ containing the field of definition of $T(\pi)$ and the roots of unity $\mu_{|T|}$ and $Q_p(x) = 1+ x\C[x]$ for all but finitely many places.

Let $S$ be a finite set of places containing all $p\mid |T|\infty$, $p$ ramified in $F/K$, and $p$ ramified in $\pi$. For $p\not\in S$, $p$ is not ramified in $\pi$ and $p$ is at most tamely ramified so that
\begin{align*}
\nu_p(\disc_\pi(f)) &= \nu_p\left(\disc(f*\pi)\right)\\
&= \ind((f*\pi)(\tau_p))\\
&= \ind(f(\tau_p))\,.
\end{align*}
In particular, this implies that $\nu_p(\disc_\pi(f_p))$ depends only on $f(\tau_p)\in T$. Under the isomorphism given in Lemma \ref{lem:KC}, if $f_\sigma\in Z^1(K_\sigma,T)$ is the isomorphic image of $f_p$ then it follows that $\nu_p(\disc_\pi(f_p))$ depends only on $f_\sigma(\tau)$. For each $p\not\in S$ with $\Leg{F/K}{p}=\sigma$ this implies
\[
Q_p(x) = \frac{1}{|T|}\sum_{f_\sigma\in Z^1(K_\sigma,T(\pi))} x^{\ind(f_\sigma(\tau))}\,,
\]
i.e. it is Frobenian in $F/K$.

All that remains is to consider the constant term, which is given by
\begin{align*}
\frac{1}{|T|} \sum_{\substack{f_p\in Z^1_{ur}(K_p,T(\pi))}} 1 &= \frac{|Z^1(\langle \Fr_p\rangle,T(\pi))|}{|T|}=1\,.
\end{align*}

Thus, Proposition \ref{prop:GeneralDirichlet} implies
\[
\prod_p Q_p(\mathcal{N}_{K/\Q}(p)^{-s}) = \zeta_K(a(Q)s)^{b(Q)}G(s)\,,
\]
where $G(s)$ is holomorphic on ${\rm Re}(s)\ge 1/a(Q)$. We note that $Q_p(x)$ has all nonnegative rational coefficients, which implies that it has no zeroes on the positive real line. In particular, $Q_p(\mathcal{N}_{K/\Q}(p)^{-1/a(Q)})\ne 0$ for all $p$, which implies $G(1/a(Q))\ne 0$ and we have identified the rightmost pole as being $s=1/a(Q)$ of order $b(Q)$. It now suffices to show that these agree with $a(T)$ and $b(K,T(\pi))$.

The $a$-invariant is easier, so we address that one first. For all $p\not\in S$, the smallest power of $\mathcal{N}_{K/\Q}(p)^{-s}$ that occurs in $Q_p(\mathcal{N}_{K/\Q}(p)^{-s})$ is given by $f_p$ such that $\nu_p(\disc_\pi(f_p))$ is nonzero and minimized. In other words, this implies
\begin{align*}
a(Q) &= \min_{\sigma\in \Gal(F/K)} - \deg(Q_\sigma(1/x)-1)\\
&= \min_{\sigma\in \Gal(F/K)} \min_{f\not\in Z^1_{ur}(K_\sigma,T(\pi))} \nu_p(\disc_\pi(f))\\
&= \min_{\sigma\in \Gal(F/K)} \min_{f\not\in Z^1_{ur}(K_\sigma,T(\pi))} \ind(f(\tau))\,.
\end{align*}
For any $t\in T$, choose $\sigma=1\in \Gal(F/K)$ and $p\not\in S$ such that $\Fr_p(F/K)=\sigma$ (up to conjugation). Define a crossed homomorphism $f:G_{1}\rightarrow T$ given by $\Fr\mapsto 1$ and $\tau\mapsto t$. Noting that $p$ is unramified in $\pi$, this is well-defined if and only if it respects the only relation
\[
c_{\Fr}(\tau) = \tau^{1}\,.
\]
This follows by construction:
\begin{align*}
f(c_{\Fr}(\tau)) &= f(\Fr\cdot \tau\cdot \Fr^{-1})\\
&= f(\Fr) c_{\pi(\Fr_p)}(f(\tau)) c_{\pi(\Fr_p\tau_p)}f(\Fr^{-1})\\
&= c_{\sigma}(f(\tau))\\
&= f(\tau)\,.
\end{align*}
This implies that for any $t\in T-\{1\}$, there exists at least one $\sigma$ such that we can attain $\ind(f(\tau))=\inf(t)$. Therefore
\[
a(Q) = \min_{t\in T-\{1\}} \ind(t) = a(T)\,.
\]

The $b$-invariant is a little more involved. We are given
\[
b(Q) = \frac{1}{[F:K]}\sum_{\sigma\in \Gal(F/K)} \frac{\#\{f\in Z^1(K_\sigma,T(\pi)) : \ind(f(\tau))=a(T)\}}{|T|}\,.
\]
Again, we know that $f\in Z^1(K_\sigma,T(\pi))$ is given by images for $\Fr,\tau\in G_m$ if and only if it respects the relationship
\[
c_{\Fr}(\tau)=\tau^{\mathcal{N}_{K/\Q}(p)}\,.
\]
Noting that $p$ unramified in $\pi$ implies $f|_{\langle \tau\rangle}$ is a homomorphism, we can equivalently check that
\[
f(\Fr) c_{\pi(\Fr_p)}(f(\tau)) c_{\pi(\Fr_p\tau_p)}(f(\Fr^{-1})) = f(\tau)^{\mathcal{N}_{K/\Q}(p)}\,.
\]
Applying the rule $f(x^{-1}) = c_{\pi(x^{-1})}(f(x)^{-1})$ for crossed homomorphisms and noting that the cyclotomic character $\chi:G_K\rightarrow \hat{\Z}^{\times}$ satisfies $\chi(\Fr_p)=\mathcal{N}_{K/\Q}(p)$ implies that it suffices to check
\[
f(\Fr) c_{\pi(\Fr_p)}(f(\tau)) c_{\pi(\Fr_p\tau_p\Fr_p^{-1})}(f(\Fr)^{-1}) = f(\tau)^{\chi(\Fr_p)}\,.
\]
Lastly, we notice that $p$ unramified in $\pi$ implies $\pi(\Fr_p \tau_p \Fr_p^{-1})=1$. Therefore it suffices to check that
\[
f(\Fr) c_{\pi(\Fr_p)}(f(\tau)) f(\Fr)^{-1} = f(\tau)^{\chi(\Fr_p)}\,,
\]
or equivalently
\[
[f(\Fr),c_{\pi(\Fr_p)}(f(\tau))]=f(\tau)^{\chi(\Fr_p)}c_{\pi(\Fr_p)}(f(\tau)^{-1})\,.
\]
Set $A(T) = \{t\in T : \ind(t)=a(T)\}$ and take a sum over all possible choices for $f(\tau)=t$ with $\ind(t)=a(T)$ and $f(\Fr)=y$ satisfying this relationship. This gives
\begin{align*}
b(Q) = \frac{1}{[F:K]} \sum_{\sigma\in \Gal(F/K)} \sum_{\substack{t\in A(T)}}\sum_{\substack{y\in T\\ [y,c_{\pi(\sigma)}(t)]=t^{\chi(\sigma)}c_{\pi(\sigma)}(t^{-1})} }\frac{1}{|T|}\,.
\end{align*}
First summing over all $y$, there are two possibilities. Either there are no values of $y\in T$ such that $[y,c_{\pi(\sigma)}(t)]=t^{\chi(\sigma)}c_{\pi(\sigma)}(t^{-1})$, or there are exactly $|C_T(c_{\pi(\sigma)}(t))|$ of them, lying in some coset of the centralizer $C_T(c_{\pi(\sigma)}(t))=\{y\in T \mid [y,c_{\pi(\sigma)}(t)]=1\}$. Therefore we can write
\begin{align*}
b(Q) = \frac{1}{[F:K]}\sum_{\sigma\in \Gal(F/K)} \sum_{\substack{t\in A(T)\\ t^{\chi(\sigma)}c_{\pi(\sigma)}(t^{-1})\in [T,c_{\pi(\sigma)}(t)]} }\frac{|C_T(c_{\pi(\sigma)}(t))|}{|T|}\,.
\end{align*}
Relabeling the summation via the automorphism $t\mapsto c_{\pi(\sigma^{-1})}(t)$ yields
\begin{align*}
b(Q) &= \frac{1}{[F:K]}\sum_{\sigma\in \Gal(F/K)} \sum_{\substack{t\in A(T)\\ c_{\pi(\sigma^{-1})}(t^{\chi(\sigma)})t^{-1}\in [T,t]} }\frac{|C_T(t)|}{|T|}\,.
\end{align*}
We now transition to conjugacy classes. The condition $c_{\pi(\sigma^{-1})}(t^{\chi(\sigma)})t^{-1}\in [T,t]$ is equivalent to $c_{\pi(\sigma^{-1})}(t^{\chi(\sigma)})$ being conjugate to $t$. Moreover, $\frac{|C_T(t)|}{|T|}$ is equal to $\frac{1}{|C|}$ for $C\subset T$ the conjugacy class of $t$. The (right) action by $\chi*\pi^{-1}$ factors through conjugacy classes, so this summation reduces to a summation over conjugacy classes $C\subset T$ containing an element of $A(T)$. The index $\ind(x)$ is invariant under conjugation, which implies the sum reduces to a summation over conjugacy classes $C\subset A(T)$.
\begin{align*}
b(Q) &= \frac{1}{[F:K]}\sum_{\sigma\in \Gal(F/K)} \sum_{\substack{C\subset A(T)\\ c_{\pi(\sigma^{-1})}(C^{\chi(\sigma)})=C} }1\\
&= \sum_{C\subset A(T)}\frac{1}{[F:K]}\#\{\sigma\in \Gal(F/K) \mid c_{\pi(\sigma^{-1})}(C^{\chi(\sigma)})=C\}\\
&= \sum_{C\subset A(T)} \frac{|{\rm Stab}_{\Gal(F/K)}(C)|}{|\Gal(F/K)|}\,,
\end{align*}
where the stabilizer is under the (right) action $\chi*\pi^{-1}$, or equivalently the (left) action $\pi*\chi^{-1}$. The orbit stabilizer formula implies that
\[
\frac{|{\rm Stab}_{\Gal(F/K)}(C)|}{|\Gal(F/K)|} = \frac{1}{\#(\text{orbit of conj. class }C\subset A(T)\text{ under }\pi*\chi^{-1})}\,,
\]
so that we have shown
\begin{align*}
b(Q) &= \sum_{C\subset A(T)} \frac{1}{\#(\text{orbit of conj. class }C\subset A(T)\text{ under }\pi*\chi^{-1})}\\
&=\#\left\{\text{orbits of }\pi*\chi^{-1}\text{ on }\{\text{conjugacy classes }C\subset A(T)\}\right\}\\
&= \#\left(\{\text{conjugacy classes }C\subset A(T)\}/\pi*\chi^{-1}\right)\\
&=b(K,T(\pi))\,.
\end{align*}
\end{proof}

\subsection{The First Cohomology Group}

Through the course of this paper we will consider towers as both corresponding to crossed homomorphisms $Z^1(K,T(\pi))$ as well as $1$-coclasses $H^1(K,T(\pi))$. The two perspectives have different benefits, but are virtually equivalent:

\begin{lemma}\label{lem:pidisc}
Let $G\subset S_n$ be a transitive subgroup and $\pi:G_K\rightarrow G$ be a homomorphism.
\begin{enumerate}
\item[(i)]{The $\pi$-discriminant factors through the Galois cohomology group $H^1(K,T(\pi))$,}

\item[(ii)]{``surjectivity'' factors through the Galois cohomology group $H^1(K,T(\pi))$, i.e. if $f,f'\in Z^1(K,T(\pi))$ satisfy $[f]=[f']$ (where $[f]$ denotes the equivalence class of $f$ in $H^1(K,T(\pi))$), and $f*\pi$ is surjective then $f'*\pi$ is also surjective,}

\item[(iii)]{
$\#\{f\in Z^1(K,T(\pi);X) : f*\pi \text{ is surjective}\}= |T/T^G| \cdot \#\{[f]\in H^1(K,T(\pi);X) : f*\pi \text{ is surjective}\}$.
}
\end{enumerate}
\end{lemma}

\begin{proof}[Proof of Lemma \ref{lem:pidisc}]
For part (i), suppose $f\in Z^1(K,T(\pi))$ and $t\in T$. Denote by $f'\in Z^1(K,T(\pi))$ the crossed homomorphism sending $x\mapsto t f(x) c_{\pi(x)}(t^{-1})$, which is equivalent to $f$ under the coboundary relation. Then for any $x\in G_K$, it follows that
\begin{align*}
(f'*\pi)(x) &= t f(x) c_{\pi(x)}(t^{-1})\pi(x)\\
&= t f(x)\pi(x) t^{-1}\\
&= c_{t}((f*\pi)(x)).
\end{align*}
This implies $f'*\pi$ is conjugate to $f*\pi$. In particular, $(f'*\pi)(G_{\mathfrak{p},i})$ is conjugate to $(f*\pi)(G_{\mathfrak{p},i})$ for all higher ramification groups. The number of orbits of an element in $S_n$ is determined by its cycle type, which is invariant under conjugation. The number of orbits of a group $H\le S_n$ is invariant under conjugation by the same argument. The discriminant $\disc(f*\pi)$ is determined by the number of orbits of $(f*\pi)(G_{\mathfrak{p},i})$ for all higher ramification groups, which we determined is independent of the coboundary relation. This implies $\disc(f*\pi)$ factors through $H^1(K,T(\pi))$, concluding the proof.

For part (ii), we remark that the image $(f*\pi)(G_K)=G$ is invariant under conjugation, which implies it is invariant under the coboundary relation. This implies that surjectivity of $f*\pi$ is a well-defined property of a $1$-coclass $[f]\in H^1(K,T(\pi))$.

For part (iii), it now suffices to show that the equivalence classes of surjective maps under the coboundary relation all have size $|T/T^G|$. This is immediate from the standard group theory fact that the map $T\rightarrow \Aut(G)$ sending $t\mapsto c_t$ the automorphism by conjugation is a homomorphism with kernel $T^G$, and that a surjective map $f*\pi$ satisfies
\[
(f*\pi)(x) = c_t((f*\pi)(x))
\]
for all $x\in G_K$ if and only if $g=c_t(g)$ for all $g\in G$.
\end{proof}

The first cohomology group also comes with restriction maps
\[
\res_p:H^1(K,T(\pi))\rightarrow H^1(K_p,T(\pi))\,,
\]
so one could reasonably ask if we should consider the following local series
\[
\prod_p \frac{1}{|H^0(K_p,T(\pi))|} \left(\sum_{f_p\in H^1(K_p,T(\pi))} \mathcal{N}_{K/\Q}(\disc_\pi(f_p))^{-s}\right)
\]
instead of the one over crossed homomorphisms for generalizing the Malle-Bhargava principle.

In general, the coboundary relation does not give equivalence classes of the same size. This means that in the local series over $H^1$, the $f_p$ may be over- or under-counted depending on how conjugation acts on the image $(f_p*\pi)(G_{K_p})$. This is a departure from the global behavior, as when counting $N(L/K,T\normal G;X)$ we do not have a reason to believe that certain local behaviors are counted differently depending on the image of the decomposition group. For general groups $T$ this may make it difficult to use powerful local-to-global results in Galois cohomology in order to verify the prediction given by the Malle-Bhargava principle.

In Sections \ref{sec:AsymptoticWiles} and \ref{sec:counting} we will specifically consider the case when $T$ is abelian by using powerful cohomological results. One of the essential reasons our techniques work for $T$ abelian is that the $Z^1$ and $H^1$ perspectives are exactly equivalent (up to a constant multiple), allowing us to prove results over $H^1$ and immediately conclude the same results over $Z^1$. This is a consequence of the equivalence of local series, which we prove below:

\begin{proposition}\label{prop:localZ1toH1}
If $T$ is abelian, then
\begin{enumerate}
\item[(i)]{$|Z^1(K,T(\pi);X)| = |T/T^G| \cdot |H^1(K,T(\pi);X)|$,
}

\item[(ii)]{
\begin{align*}
&\prod_p \frac{1}{|T|}\left(\sum_{f_p\in Z^1(K_p,T(\pi))} \mathcal{N}_{K/\Q}(\disc_\pi(f_p))^{-s}\right)\\
&=\prod_p \frac{1}{|H^0(K_p,T(\pi))|} \left(\sum_{f_p\in H^1(K_p,T(\pi))} \mathcal{N}_{K/\Q}(\disc_\pi(f_p))^{-s}\right)\,.
\end{align*}
}
\end{enumerate}
\end{proposition}

\begin{proof}
For $T$ abelian, the coboundary relation is given by the quotient relation by the group of $1$-coboundaries. Part (i) follows from $|B^1(K,T(\pi))| = |T(\pi)/(T(\pi))^G|=|T/T^G|$. For part (ii), we utilize the fact that $|B^1(K_p,T(\pi))|=|(T(\pi))/(T(\pi))^{G_{K_p}}|$. This implies
\begin{align*}
&\prod_p \frac{1}{|T|}\left(\sum_{f_p\in Z^1(K_p,T(\pi))} \mathcal{N}_{K/\Q}(\disc_\pi(f_p))^{-s}\right)\\
&= \prod_p \frac{|(T(\pi))/(T(\pi))^{G_{K_p}}|}{|T|} \left(\sum_{f_p\in H^1(K_p,T(\pi))} \mathcal{N}_{K/\Q}(\disc_\pi(f_p))^{-s}\right)\\
&=\prod_p \frac{1}{|(T(\pi))^{G_{K_p}}|} \left(\sum_{f_p\in H^1(K_p,T(\pi))} \mathcal{N}_{K/\Q}(\disc_\pi(f_p))^{-s}\right)\\
&=\prod_p \frac{1}{|H^0(K_p,T(\pi))|} \left(\sum_{f_p\in H^1(K_p,T(\pi))} \mathcal{N}_{K/\Q}(\disc_\pi(f_p))^{-s}\right)\,.
\end{align*}
\end{proof}

\subsection{The Inconsistency of the Malle-Bhargava Principle}

After seeing how nicely the Malle-Bhargava principle generalizes to counting $(T\normal G)$-towers, we now turn to the known inconsistencies of the principle. Kl\"uners demonstrated this by showing that $b(K,C_3\wr C_2)$ is the wrong value \cite{kluners2005}. Kl\"uners' paper is not long, and does not dwell too much on what is causing the problem. Kl\"uners limits his insight into the issue to essentially stating that too many roots of unity can cause problems, and lists a family of groups where one should expect it to cause problems by a similar argument to what he uses for $C_3\wr C_2$.

T\"urkelli \cite{turkelli2015} proposed the first, and to the author's knowledge only, correction to Malle's conjecture which accounts for Kl\"uners' counter example. T\"urkelli proposes that we should instead have $\log X$ to the power of
\[
B(K,G) = \max_{\phi} b_{\phi}(K,G),
\]
where $b_\phi(K,G)$ is the usual invariant from Malle's conjecture, except instead of modding out by the conjugation action and the cyclotomic action you mod out by the twisted action $\phi*\chi^{-1}$ where $\phi\in \Hom(G_K^{ab},G/N)$ for any normal subgroup $N\normal G$ containing the commutator of $G$ with $a(N) = a(G)$. This is the same twisted action we see popping up in the Malle-Bhargava principle for $N(L/K,T\normal G;X)$ and $|Z^1(K,T(\pi);X)|$, where we twist by some homomorphism $\pi$ inducing the Galois action on $T$. T\"urkelli justifies this modification by appealing to the function field analog, where he proves that this is the correct notion for $G = T\rtimes B$ when $B$ is cyclic containing no nontrivial normal subgroups of $G$ and $|G|$ is prime to the characteristic. In this case, $B(K,G)$ is found to be the number of connected components of a certain Hurwitz scheme.

We like to think of function fields as being an ``easier version" of number fields in which all the same statements are generally true, but there is usually no concrete way to take a proof in the function field case and translate in into a proof in the number field case. In particular, subtle differences between abelian extensions of function fields and number fields make it more difficult to use T\"urkelli's results to justify the modification in the number field case. Corollary \ref{cor:lowerbound} is the first known result to give strong evidence that T\"urkelli's modification is correct for number fields specifically, or is at least moving in the right direction.

We can show more theoretic evidence that T\"urkelli's modification is correct by considering a more general inconsistency with the Malle-Bhargava principle on towers, of which Kl\"uners' counter example is a special case:

\begin{proposition}\label{prop:differentpi}
There exist transitive subgroups $G\subset S_n$ with $T\normal G$, $A(G)\subset T$, $(L/K,\iota_B)$ a $B$-extension, and $\pi:G_K\rightarrow G$ a $(T\normal G)$-tower such that
\[
b(K,T(\pi)) > b(K,G)
\]
but
\[
N(L/K,T\normal G;X) \le N(K,G;X)\,.
\]
This occurs exactly when the action
\[
\pi*\chi^{-1}:G_K\rightarrow \Aut(\{\text{conjugacy classes }C\subset A(T)\})
\]
satisfies $(\pi*\chi^{-1})(G_K) \ne \pi(G_K)\chi(G_K)^{-1}$.
\end{proposition}

This proposition follows immediately from the definitions of $a(T)$, $a(G)$, $b(K,T(\pi))$, and $b(K,G)$, and shows that the Malle-Bhargava principle is not consistent with partitioning by subfields. One might take this to mean that our application of the Malle-Bhargava principle to crossed homomorphisms is the cause of the inconsistency, but Kl\"uners' counter example shows that this is not the case. There are groups $G$ for which the predicted invariant $b(K,G)$ in the original Malle-Bhargava principle is wrong, but $b(K,T(\pi))$ is correct for an appropriate choice of subgroup $T\normal G$ and homomorphism $\pi\in \Hom(G_K,G)$.

The refined counting problem is seeing a possible overlap between conjugation (the action of $\pi$) and the cyclotomic action (the action of $\chi$), which Malle's original prediction treats as being independent of each other. In this way, the refined counting problem sheds light on the known inconsistencies of Malle's conjecture.

\textbf{Example:} Consider Kl\"uners' counter example $G=C_3\wr C_2\subset S_6$ and the abelian normal subgroup $T=C_3^2\normal G$. Any surjective homomorphism $\pi:G_K\rightarrow C_2$ is \emph{itself} a solution to the embedding problem because $G$ is split, so it follows from Corollary \ref{cor:alpha} that
\[
N(L/\Q,C_3^2\normal C_3\wr C_2;X) \sim c'(K,T(\pi)) X^{1/a(T)}(\log X)^{b(K,T(\pi))-1}.
\]
The group $T$ is the subgroup of permutations in $S_6$ generated by $(1\ 2\ 3)$ and $(4\ 5\ 6)$. This implies $a(T) = 2$.

The set of elements of $T$ with minimal index is exactly
\[
A(T)=\{(1\ 2\ 3),(1\ 2\ 3)^2,(4\ 5\ 6),(4\ 5\ 6)^2\}.
\]
We have two cases for computing $b(\Q,T(\pi))$:
\begin{enumerate}
\item{If $L\ne \Q(\zeta_3)$, then the composite map $\pi*\chi^{-1}:G_\Q\rightarrow \textnormal{Sym}(A(T))$ acts transitively on $A(T)$, as we can find $\sigma$ such that $\pi(\sigma)\ne 1$ and $\chi(\sigma)=1$ so that $\sigma.(1\ 2\ 3) = (4\ 5\ 6)$, and vice versa with $\pi(\sigma)=1$ and $\chi(\sigma)\ne 1$ so that $\sigma.(1\ 2\ 3) = (1\ 2\ 3)^2$. This implies there is a single orbit.}
\item{If $L=\Q(\zeta_3)$, the opposite is true and $\pi(\sigma) = 1$ if and only if $\chi(\sigma)=1$. This implies $\sigma.(1\ 2\ 3)$ is either $(1\ 2\ 3)$ or $(4\ 5\ 6)^2$, which partitions $A(T)$ into two orbits.}
\end{enumerate}
Therefore
\[
b(\Q,T(\pi)) = \begin{cases}
2 & L=\Q(\zeta_3)\\
1 & L\ne \Q(\zeta_3).
\end{cases}
\]
On the other hand, $A(G)=A(T)$ and $a(G)=a(T)$ is made up of two conjugacy classes $\{(1\ 2\ 3), (4\ 5\ 6)\}$ and $\{(1\ 2\ 3)^2, (4\ 5\ 6)^2\}$, which are swapped by the cyclotomic action. This implies
\[
b(\Q,G) = 1.
\]
The Malle-Bhargava principle then predicts
\begin{align*}
c' X^{1/2}\log X \sim N(\Q(\zeta_3)/\Q,C_3^2\normal C_3\wr C_2;X) \le N(\Q,C_3\wr C_2;X) \sim c X^{1/2}\,,
\end{align*}
which is a clear contradiction.

Step 2 as described in the introduction can be made to work in this case, and shows that $N(\Q,G;X) \sim c X^{1/2}\log X$ is the correct value (see \cite{lemke-oliver-jwang-wood2019} for a proof of uniformity in this case). The original power predicted by Malle was $0$, but Kl\"uners showed it must be at least $1$ and T\"urkelli's modification supports that it should be exactly $1$. Subject to uniformity, the study of $(T\normal G)$-towers suggests the same asymptotic as T\"urkelli.

We can generalize T\"urkelli's modification to $(T\normal G)$-towers as follows:

\begin{definition}
We define T\"urkelli's modified invariant for $T\normal G$ and $\pi:G_K\twoheadrightarrow G$ to be
\[
B(K,T(\pi)) = \max_{\substack{N\normal G\\N\normal T\\a(N)=a(T)} }\max_{\substack{\varphi:G_K\rightarrow G\\ \varphi \equiv \pi \mod N} }b(K,N(\varphi))\,.
\]
\end{definition}

T\"urkelli's definition included the extra condition that $[T,T]\le N$. This is not necessary to state, as any $N$ for which this is not true will yield a smaller invariant than $N[T,T]$.

\begin{lemma}
Suppose $T\normal G$, $N\normal G$ with $a(N)=a(T)$ and $N\subset T$, and $\pi:G_K\twoheadrightarrow G$. Then for any $\varphi\equiv \pi\mod N$
\[
b(K,N(\varphi))\le b(K,N(T\cap [G,G])(\varphi))\,.
\]
\end{lemma}

\begin{proof}
The cyclotomic character factors through $G_K^{ab}$, which implies $\varphi*\chi^{-1}|_{[G_K,G_K]}=\varphi|_{[G_K,G_K]}$. $a(N)=a(T)$ implies that $A(N)=A(T) \cap N$, and similarly $A(N(T\cap [G,G]))=A(T)\cap N(T\cap[G,G])$.

Suppose $C_1,C_2\subset A(N)$ are conjugacy classes in $N$ such that there exists some $x\in N(T\cap[G,G])$ with $c_x(C_1)=C_2$. The action factors through $G/N$ as $N$ acts trivially on $N$-conjugacy classes, so this implies there exists some $x\in N(T\cap [G,G])/N$ with $c_x(C_1)=C_2$. Surjectivity of $\pi$ and $\varphi\equiv \pi \mod N$ implies that
\begin{align*}
N(T\cap [G,G])/N &= N(T\cap \pi([G_K,G_K]))/N\\
&= N(T\cap \varphi([G_K,G_K]))/N\\
&= N(T\cap (\varphi*\chi^{-1})([G_K,G_K]))/N\,.
\end{align*}
Thus we have shown that the conjugacy class $\widetilde{C_1}\subset A(N(T\cap [G,G]))$ containing $C_1\subset A(N)$ is necessarily contained in the union
\[
\bigcup_{\sigma\in G_K} c_{\varphi(\sigma)}(C_1^{\chi(\sigma)^{-1}})\,.
\]
This proves that the map
\[
\{\text{conjugacy class } C\subset A(N)\}/\varphi*\chi^{-1} \rightarrow \{\text{conjugacy class }\widetilde{C}\subset A(N(T\cap[G,G]))\}/\varphi*\chi^{-1}
\]
is a well-defined inclusion, which concludes the proof.
\end{proof}

We provide the statement for a generalized form of T\"urkelli's modification for $(T\normal G)$-towers here:

\begin{conjecture}[T\"urkelli's modification for $(T\normal G)$-towers]
Let $G\subset S_n$ be transitive, $T\normal G$, $(L/K,\iota_B)$ a $B$-extension and $\pi:G_K\rightarrow G$ a homomorphism with $\pi\equiv \iota_B \mod T$. Then
\[
N(L/K,T\normal G;X) \sim c'(K,T(\pi)) X^{1/a(T)} (\log X)^{B(K,T(\pi))-1}\,.
\]
\end{conjecture}

T\"urkelli's modification is a brute force fix for the issues arising in Proposition \ref{prop:differentpi} It is the maximum over all the possible $b(K,T(\pi))$ that could provide a counter example as in Proposition \ref{prop:differentpi}, which removes the inconsistency.

All of the cases given in the introduction for which Malle's conjecture proper has been proven satisfy $b(K,G) = B(K,G)$ (although the upcoming preprint \cite{lemke-oliver-jwang-wood2019} will include the proofs of at least one case with $b(K,G)<B(K,G)$). Corollary \ref{cor:alpha} proves the generalized Malle's conjecture for towers with $b(K,T(\pi))$ for $T$ abelian, which would support T\"urkelli's modification if we can show $b(K,T(\pi))=B(K,T(\pi))$ for $T$ abelian.

\begin{proposition}
Suppose $T\normal G$ and $\pi:G_K\twoheadrightarrow G$. Then the following are true:
\begin{enumerate}
\item[(i)]{
If $A(T)\subset Z(T)$ is contained in the center of $T$ then $b(K,T(\pi))=B(K,T(\pi))$.
}

\item[(ii)]{
If $T\le [G,G]$ then $b(K,T(\pi)) = B(K,T(\pi))$.
}
\end{enumerate}
\end{proposition}

This proposition shows that Corollary \ref{cor:alpha} verifies T\"urkelli's modification in the case that $T$ is abelian, as $Z(T)=T$ in that case. We also include a large family of other cases for which T\"urkelli's modification does not changes the log term, which highlights the fact that T\"urkelli's modification is really about issues arising from abelian extensions and roots of unity. This remark gives evidence suggesting that Malle's original conjecture should hold for groups with trivial abelianization, in particular for all nonabelian simple groups.

\begin{proof}
For part (i), every $N\le T$ with $a(N)=a(T)$ satisfies $A(N)\subset A(T) \subset Z(T)$ which implies $T$ acts on $A(T)$ trivially by conjugation. Thus
\begin{align*}
B(K,T(\pi)) &= \max_{\substack{N\normal G\\N\normal T\\a(N)=a(T)} }\max_{\substack{\varphi:G_K\rightarrow G\\ \varphi \equiv \pi \mod N} }b(K,N(\varphi))\\
&= \max_{\substack{N\normal G\\N\normal T\\a(N)=a(T)} }\max_{\substack{\varphi:G_K\rightarrow G\\ \varphi \equiv \pi \mod N} }\#\left(\{\text{conjugacy class } C\subset A(N)\}/\varphi*\chi^{-1}\right)\,.
\end{align*}
All conjugacy classes in $A(N)\subset Z(T)$ are trivial, so this really only considers orbits of elements.
\begin{align*}
B(K,T(\pi)) &= \max_{\substack{N\normal G\\N\normal T\\a(N)=a(T)} } \max_{\substack{\varphi:G_K\rightarrow G\\ \varphi \equiv \pi \mod N} }\#\left(A(N)/\varphi*\chi^{-1}\right)\,.
\end{align*}
The action by conjugation of $G$ on $A(N)$ factors through $G/T$ because $A(N)\subset Z(T)$. Therefore $\varphi\equiv \pi\mod N$ implies $\varphi\equiv \pi\mod T$ implies $\varphi$ and $\pi$ induce the same action on $A(N)$. This implies
\begin{align*}
B(K,T(\pi)) &= \max_{\substack{N\normal G\\N\normal T\\a(N)=a(T)} } \#\left(A(N)/\pi*\chi^{-1}\right)\\
&=\#\left(A(T)/\pi*\chi^{-1}\right)\\
&=b(K,T(\pi))\,.
\end{align*}

Part (ii) follows from the bound
\begin{align*}
b(K,N(\varphi)) &\le b(K,N(T\cap [G,G])(\varphi))\\
&= b(K,T(\varphi))\,,
\end{align*}
so that
\begin{align*}
B(K,T(\pi)) &=  \max_{\substack{N\normal G\\N\normal T\\a(N)=a(T)} }\max_{\substack{\varphi:G_K\rightarrow G\\ \varphi \equiv \pi \mod N} }b(K,T(\varphi))\,.
\end{align*}
We remark that $\varphi\equiv \pi\mod N$ in particular implies that $\varphi \equiv \pi\mod T$, which implies that they act on $\{\text{conjugacy class }C\subset A(T)\}$ in the same way. Therefore
\begin{align*}
B(K,T(\pi)) &=  \max_{\substack{N\normal G\\N\normal T\\a(N)=a(T)} }b(K,T(\pi))\\
&= b(K,T(\pi))\,.
\end{align*}
\end{proof}

\newpage

\section{Asymptotic Wiles' Theorem}\label{sec:AsymptoticWiles}

\subsection{A Review of Wiles' Theorem}

This subsection is dedicated to reviewing the key points of Wiles' Theorem on generalized Selmer groups, which Wiles proves in \cite{wiles1995} as a part of his proof of modularity. Section 2.3 of Darmon-Diamond-Taylor \cite{darmon-diamond-taylor1995} is a useful survey of the results discussed in this section for the reader interested in more details and how Wiles uses this result in his proof of modularity. The reader who is familiar with these results may skip ahead to the next subsection.

We define, as Wiles does, a kind of Selmer group for an arbitrary Galois module:
\begin{definition}\label{def:Selmer}
Fix a Galois module $T$ (i.e. an abelian group with a Galois action), and a family $\mathcal{L}=(L_p)$ of subgroups $L_p\le H^1(K_p,T)$ at all places $p$ of $K$. Define the \textbf{generalized Selmer group} associated to $\mathcal{L}$ to be
\[
H^1_{\mathcal{L}}(K,T) = \left\{ f\in H^1(K,T) \mid \forall p,\res_p(f)\in L_p\right\}.
\]
In other words, this is the preimage of $\prod L_p$ under the restriction map.
\end{definition}

Wiles' key observation was that there is a very close relationship between the Selmer group and the corresponding dual Selmer group under the Tate pairing. Recall that the Tate pairing is a perfect pairing for each place $p$
\[
H^1(K_p,T) \times H^1(K_p,T^*) \rightarrow \mu_{|T|}\,,
\]
where $T^*=\Hom(T,\mu_{|T|})$ is the dual Galois module to $T$. For any subgroup $N\le H^1(K_p,T)$ we can define the Tate dual to $N$ by $N^*=\textnormal{Ann}(N)$ the annihilator of $N$ under the Tate pairing. The dual family to $\mathcal{L}$ is then given by $\mathcal{L}^*=(L_p^*)$, and the dual Selmer group is the Selmer group corresponding to $\mathcal{L}^*$ on the Galois module $T^*$.

Wiles notes that if the Selmer group is ``unramifed" away from finitely many places, then it is finite and he was able to give a formula for the size. Define $H^1_{ur}(K_p,T) := H^1(G_{K_p}/I_p,T^{I_p})$ to be the kernel of the restriction map to $H^1(I_p,T)$, so that we can explicitly say $f\in H^1(K,G)$ is unramified at $p$ if $\res_p(f)\in H^1_{ur}(K_p,T)$.

\begin{theorem}[Wiles' theorem]\label{thm:Wiles}
Let $\mathcal{L}$ be as in Definition \ref{def:Selmer} such that $L_p=H^1_{ur}(K_p,T)$ for all but finitely many places. Then $H^1_{\mathcal{L}}(K,T)$ is finite and
\begin{align*}
\frac{|H^1_{\mathcal{L}}(K,T)|}{|H^1_{\mathcal{L}^*}(K,T^*)|} = \frac{|H^0(K,T)|}{|H^0(K,T^*)|}\prod_{p} \frac{|L_p|}{|H^0(K_p,T)|}.
\end{align*}
\end{theorem}

Wiles made use of this theorem in special cases where $K=\Q$ and the dual Selmer group was in fact trivial in order to get good sizes for the Selmer group, but the proof in general is the same via a clever use of the nine term Poitou-Tate exact sequence. This theorem is, in a certain sense, a ``local-to-global" theorem. It expresses the global quantity $|H^1_{\mathcal{L}}(K,T)|$ as (almost) a product of local densities $|L_p|/|H^0(K_p,T)|$. Galois cohomology of local fields is very well understood, so a local-to-global theorem of this kind allows us to take that information and prove new things about Galois cohomology of global fields.

What about the pieces that are not local densities? $H^0(K,T)$ and $H^0(K,T^*)$ are constants independent of $\mathcal{L}$, so they essentially do not matter. The one confounding factor comes from the dual Selmer group. We will prove an asymptotic version of Wiles' theorem for $\mathcal{L}$ that allows for ramification at infinitely many places, and we will give an argument that the dual Selmer group is not ``too bad" for large families $\mathcal{L}$ and the behavior is dominated by the product of local densities.

Here we list some useful facts about the Tate pairing and Galois cohomology that are important to Wiles' theorem and that we will utilize in this section:
\begin{itemize}
\item{If $F$ is the field fixed by the action $G_K\rightarrow \Aut(T)$ and $F^*$ is the field fixed by $G_K\rightarrow \Aut(T^*)$, then $FF^*=F(\mu_{|T|})$. This is what causes the cyclotomic character to show up in the number field counting problem, and explains the assumptions for $F$ in Lemma \ref{lem:KC}.}

\item{If $p\nmid \infty$, then $|H^1_{ur}(K_p,T)|=|H^0(K_p,T)|$. This implies that the product in Wiles' theorem is really a finite product, as all but finitely many places are unramified.}

\item{If $p\nmid |T|\infty$ then $H^1_{ur}(K_p,T)^* = H^1_{ur}(K_p,T^*)$. This tells us that, away from wild or infinite places, unramified means the same thing in the Selmer group and its dual. This will be important for controlling the size of the dual Selmer group.}

\item{$0^*=H^1(K_p,T^*)$ and $H^1(K_p,T)^*=0$. This follows from the Tate pairing being a perfect pairing.}

\item{If $N_1\le N_2 \le H^1(K_p,T)$ then $N_2^*\le N_1^* \le H^1(K_p,T^*)$. This follows from a manipulation of annihilators of pairings.}
\end{itemize}
These facts can be found in either Darmon-Diamond-Taylor \cite{darmon-diamond-taylor1995} or any text on local Galois cohomology.

\textbf{What direction will we be going in?} $H^1(K,T)$ is technically a Selmer group, just where $L_p=H^1(K_p,T)$ for all places $p$. However, it does not satisfy the hypothesis that all but finitely many places are unramified and in general it is not finite. We will provide a partial answer to the following general question:

\begin{question}
What is the ``size" of $H^1_{\mathcal{L}}(K,T)$? In particular, if $H^1_{\mathcal{L}}(K,T)$ is infinite how is it distributed?
\end{question}

We will do that by ordering the 1-coclasses with some kind of discriminant-like invariant, and counting asymptotically how large the set of 1-coclasses in $H^1_{\mathcal{L}}(K,T)$ with discriminant $<X$ is as we take $X\rightarrow \infty$.

\subsection{Frobenian local conditions and discriminants}\label{subsec:frobLdisc}

Specifying $\mathcal{L}$ is equivalent to specifying infinitely many local conditions for 1-coclasses to satisfy. Specifying infinitely many local conditions does not always yield something ``nice" to count.  A good example of this would be to try and specify the splitting behavior of all places in an extension $L/K$. It's difficult to say whether there are any fields that satisfy a prescribed splitting behavior at each place, and in many cases there will not be. In fact, there are uncountably many ways to specify a splitting type for all places, but only countably many extensions $L/K$!

When specifying local conditions, typically the ones that are ``nice" to deal with are the ones that are distributed like the splitting behavior of places in a finite extension $L/K$. For the purposes of this paper, we really want ``nice" to mean that the corresponding Dirichlet series has a meromorphic continuation with an easily described rightmost pole. Up to taking a branch cut, we can do this when the local behavior is distributed like the splitting type of places in a finite extension $L/K$.

This was the subject of Section \ref{sec:analytic}, where we developed the necessary analytic tools in exactly the case that the local factors are Frobenian. Proposition \ref{prop:GeneralDirichlet} suggests that we consider the following definition for a ``nice" family of local conditions:

\begin{definition}
Call $\mathcal{L}$ \textbf{Frobenian in $\mathbf{F/K}$} for a finite extension $F/K$ and a finite set of places $S$ if it satisfies the following properties:
\begin{itemize}
\item[(a)]{$F$ contains the fields of definition for $T$ and $T^*$.}

\item[(b)]{$S$ contains all ramified places in $F$ and all places $p\mid |T|\infty$.}

\item[(c)]{For all $\sigma\in \Gal(F/K)$, there exists a subgroup $L_\sigma\le H^1(K_\sigma,T)$ (only defined up to conjugation on $\sigma$) such that if $p\not\in S$ with $\Leg{F/K}{p}$ conjugate to $\sigma$ then $L_p\cong L_\sigma$ under the natural isomorphism in Lemma \ref{lem:KC}.}
\end{itemize}
Call any place inside of $S$ an \textbf{irregular place}.
\end{definition}

By construction, Lemma \ref{lem:KC} implies that the local conditions $L_p=H^1(K_p,T)$ are necessarily Frobenian in an extension $F/K$ containing the field of definition of $T$ and $K(\mu_{|T|})$.

When $H_{\mathcal{L}}(K,T)$ is infinite, we describe the ``size" by fixing an admissible ordering by some invariant $\inv:H_{\mathcal{L}}^1(K,T)\rightarrow I_K$ as in \cite[Definition 2.1]{alberts2020} and describing the asymptotic growth of the sets
\[
H^1_{\mathcal{L}}(K,T;X) :=\left\{f\in H^1_{\mathcal{L}}(K,T) : \mathcal{N}_{K/\Q}(\inv(f))<X\right\}.
\]
This is motivated by more classical arithmetic statistics problems, like number field counting as in Malle's conjecture. We restate the definition of an admissible ordering here, extended to allow for nontrivial actions:
\begin{definition}[Definition 2.1 \cite{alberts2020}]
We define an \textbf{admissible ordering} (or \textbf{admissible invariant}) $\inv:\prod_{\pK}H^1(K_p,T) \rightarrow I_K$ as follows:
\begin{enumerate}
\item[(a)]{there is a family of functions $\inv_p:H^1(I_p,T)\rightarrow \Z_{\ge 0}$ for each place $p$ of $K$ such that
\[
\inv(f) = \prod_{p} \pK^{\inv_p(f|_{I_p})}\,,
\]
i.e. $\nu_p(\inv(f))=\inv_p(f)$ is determined by $f|_{I_p}$,
}

\item[(b)]{
for all but finitely many places $p$ of $K$, $f(I_\pK)=1$ if and only if $\inv_p(f)=0$.
}
\end{enumerate}
We define $\inv:H^1(K,T)\rightarrow I_K$ by $\inv(f) = \inv((f|_{G_{K_p}})_{p})$.
\end{definition}
We remark that
\[
H^1_{\mathcal{L}}(K,T;X) = \{f\in H^1_{\mathcal{L}}(K,T) : \mathcal{N}_{K/\Q}(\inv(f))<X\}
\]
is finite by Wiles theorem, as it is contained in the finite Selmer group $H^1_{\mathcal{L}'}(K,T)$ for
\[
L_p' = \begin{cases}
L_p & \mathcal{N}_{K/\Q}(p) < X\text{ or }p\mid \infty\\
H^1_{ur}(K,T) & \mathcal{N}_{K/\Q}(p)\ge X.
\end{cases}
\]
The author in \cite{alberts2020} proves upper bounds for number fields ordered by \emph{any} admissible ordering, but as discussed in that paper we do not generally expect an arbitrary admissible ordering to give a nice asymptotic main term. For a pedantic example, consider the ordering defined by
\[
\inv_p(f|_{I_p}) = \begin{cases}
\mathcal{N}_{K/\Q}(p) & f|_{I_p} \ne 1\\
0 & f|_{I_p} = 1\,.
\end{cases}
\]
The results in \cite{alberts2020} show that $\#\{f : \mathcal{N}_{K/\Q}(\inv(f))<X\} \ll X^{\epsilon}$, which behaves unlike other counting functions that appear in this setting. This is proven by showing that the corresponding Dirichlet series converges absolutely for ${\rm Re}(s)>0$, which implies that we cannot make use of a Tauberian theorem as there is no rightmost pole with positive real part.

In order to prove results about the asymptotic main term, rather than just bounds, we restrict to admissible orderings for which $\nu_p(\inv(f))$ is ``nicely distributed", i.e. so that the Euler product of local terms
\[
\prod_{p}\frac{1}{|H^0(K_p,T)|}\sum_{f\in H^1(K_p,T)} \mathcal{N}_{K/\Q}(p)^{-v_p(\inv(f))s}
\]
has a meromorphic continuation with an easily described rightmost pole.

This are all satisfied by the usual discriminant on $G$-extensions of number fields, where that last condition follows from $\nu_p(\disc(f))= n - \#\{\text{orbits of }f(I_p)\}$ whenever $G\subset S_n$ is a transitive subgroup and $p\nmid |G|\infty$ is tame. Section \ref{sec:analytic} again suggests that the local Euler products do have a nicely described rightmost pole if the local factors are Frobenian.

\begin{definition}
We say an admissible ordering $\inv:H^1(K,T)\rightarrow I_K$ is \textbf{Frobenian in $\mathbf{F/K}$} if there exists a finite set of places $S$ such that
\begin{itemize}
\item[(a)]{$F$ contains the fields of definition of $T$ and $T^*$.}

\item[(b)]{$S$ contains all places ramified in $F$ and all places $p\mid |T|\infty$.}

\item[(c)]{If $p\not \in S$ then $f\in H^1_{ur}(K,T)$ if and only if $v_p(\inv(f))=0$.}

\item[(d)]{For each $\sigma\in \Gal(F/K)$ there exists a map $v_\sigma:H^1(K_\sigma,T)\rightarrow \Z_{\ge 0}$ (only defined up to conjugation on $\sigma$) such that whenever $p\not\in S$ and $\Leg{F/K}{p}$ is conjugate to $\sigma$ then $v_p(\inv(f))=v_\sigma(\res_p(f))$, where we identify $\res_p(f)$ with its image in $H^1(K_\sigma,T)$ under the isomorphism in Lemma \ref{lem:KC}.}
\end{itemize}
Call any place inside of $S$ an \textbf{irregular place}.
\end{definition}

Examples of Frobenian orderings include the product of ramified places map
\[
\ram(f) = \prod_{p: f(I_p)\ne 1} p\,,
\]
and the $\pi$-discriminant. We will prove this for the $\pi$-discriminant in Section \ref{sec:counting}.

\subsection{The statement of the Asymptotic Wiles' Theorem}

We are now ready to state an asymptotic version of Wiles' Theorem in full detail. When we have both $\mathcal{L}$ and an ordering specified, we can talk about the elements of $L_p$ whose ordering is not too large. For an integer $m\ge 0$, define
\[
L_p^{[m]} = \{f_p\in L_p : \nu_p(\inv(f_p))= m\}.
\]
For convenience, we define $L_p^{[\infty]} = L_p\cap H^1_{ur}(K_p,T)$. The asymptotic behavior will be dominated by the minimal value of $m>0$ with $L_p^{[m]}\ne \emptyset$ for infinitely many places $p$. To that effect, we make the following definition:

\begin{definition}
Fix $T$ a finite $K$-module, and suppose $\mathcal{L}$ and $\inv$ are Frobenian in $F/K$. Then define
\begin{align*}
a_{\inv}(\mathcal{L}) =& \min_{\sigma\in \Gal(F/K)} \min_{\substack{f\in L_\sigma\\ f\not\in H^1_{ur}(K_\sigma,T)}} \nu_\sigma(\inv(f))\\
=& \textnormal{ the minimum power of a tamely ramified place }p\not\in S\\
&\textnormal{ that can occur in }\inv(f_p)\textnormal{ for }f_p\in L_p\,,
\end{align*}
where we take the convention that $\min_{n\in \emptyset}n = \infty$. Additionally define
\begin{align*}
b_\inv(\mathcal{L}) =& \frac{1}{[F:K]}\sum_{\sigma\in \Gal(F/K)} \frac{|L_\sigma^{[a_\inv(\mathcal{L})]}|}{|H^0(K_\sigma,T)|}\\
=&\textnormal{ the average size of }\frac{|L_p^{[a_{\inv}(\mathcal{L})]}|}{|H^0(K_p,T)|}\,.
\end{align*}
\end{definition}

The integer $a_\inv(\mathcal{L})$ is the minimal value of $m>0$ such that $L_p^{[m]}\ne \emptyset$ for \emph{infinitely many} places, which should remind the reader of $a(G)$ from Malle's conjecture which can be expressed as the minimum exponent that can occur in the discriminant for infinitely many places.

\begin{theorem}[Asymptotic Wiles Theorem]\label{thm:asymptoticWiles}
Let $T$ be a finite Galois module over $K$ and $\mathcal{L}$ and $\inv$ be Frobenian in $F/K$ satisfying
\begin{itemize}
\item[(a)]{$S$ is the set of irregular places,}

\item[(b)]{For all $\sigma \in \Gal(F/K)$, $H^1_{ur}(K_\sigma,T)\subset L_\sigma$,}

\item[(c)]{For all places $p$, if $f,f'\in H^1(K_p,T)$ such that $\langle f|_{I_p}\rangle = \langle f'|_{I_p}\rangle \le H^1(I_p,T)$ then $\nu_p(\inv(f))=\nu_p(\inv(f'))$.}
\end{itemize}
Then
\begin{align*}
|H^1_{\mathcal{L}}(K,T;X)| \sim c_\inv(\mathcal{L}) X^{1/a_\inv(\mathcal{L})}(\log X)^{b_{\inv}(\mathcal{L})-1},
\end{align*}
for some positive constant $c_\inv(\mathcal{L})$. (Here we take the convention that $1/a_{\inv}(\mathcal{L})=0$ if $a_{\inv}(\mathcal{L})=\infty$.)
\end{theorem}

We discussed the necessity of choosing Frobenian $\mathcal{L}$ and $\inv$, it will allow us to use Corollary \ref{cor:maintauberian} to convert analytic information at the rightmost pole of a Dirichlet series to asymptotic information. Conditions (b) and (c) are new, and it would be good to address them separately.

Condition (b) prevents us from specifying the splitting type at more than finitely many places. This avoids many issues about the existence of fields with prescribed splitting type at infinitely many places, but it may be more instructive to consider why this works well with our method. Wiles originally related the size of a Selmer group to the size of the corresponding dual Selmer group, but the dual Selmer group is seemingly nowhere to be found in Theorem \ref{thm:asymptoticWiles}. Condition (b) will force the dual Selmer group to be finite, so that $|H^1_{\mathcal{L}^*}(K,T^*)|$ is just a factor of the positive constant $c_\inv(\mathcal{L})$. We briefly prove this here:
\begin{lemma}
If $\mathcal{L}$ is as in Theorem \ref{thm:asymptoticWiles}, then $H^1_{\mathcal{L}^*}(K,T^*)$ is finite.
\end{lemma}
\begin{proof}
$H^1_{ur}(K,T)\subset L_p$ implies that $L_p^*\subset H^1_{ur}(K_p,T)^*$. Whenever $p\nmid |T|\infty$, $H^1_{ur}(K,T)^*= H^1_{ur}(K,T^*)$ so that
\[
H^1_{\mathcal{L}^*}(K,T^*) \subset H^1_{\mathcal{L}_0^*}(K,T^*),
\]
where
\[
(L_0)_p^* = \begin{cases}
L_p^* & p\in S\\
H^1_{ur}(K,T^*) & p\not\in S.
\end{cases}
\]
$S$ is a finite set, so $\mathcal{L}_0^*$ satisfies the hypotheses of Wiles' Theorem \ref{thm:Wiles}, and so $H^1_{\mathcal{L}_0^*}(K,T^*)$ must be finite.
\end{proof}

Condition (c) is a bit more subtle. One short-coming of Wiles' Theorem is that it can only deal with local conditions $\mathcal{L}=(L_p)$ for the $L_p$ given as \emph{subgroups} of $H^1(K_p,T)$, not arbitrary subsets. Due to this restriction, Wiles' Theorem will only allow us to see the ramification $f|_{I_p}$ ``up to subgroups", i.e. we will only be able to see the cyclic group $\langle f|_{I_p}\rangle$ and not the individual elements. Due to this fact, we want to make sure that the ordering is not separating the individual elements. See Subsection \ref{subsec:finordSelmer} for the place we utilize this property.

\subsection{The proof of the Asymptotic Wiles' Theorem}

We will prove this theorem by considering the Dirichlet series
\begin{align*}
H_{\mathcal{L}}(s)&:=\sum_{f\in H^1_{\mathcal{L}}(K,T)} \mathcal{N}_{K/\Q}(\inv(f))^{-s}\\
&=\sum_{a\in I_K} \#\{f\in H^1_{\mathcal{L}}(K,T): \inv(f)=a\}\mathcal{N}_{K/\Q}(a)^{-s}.
\end{align*}
The proof takes place in three parts. First, we reduce the Dirichlet series to a sum of sizes of finite order Selmer groups via an inclusion-exclusion argument. Second, we apply Wiles' theorem to decompose this series into a finite sum of Euler products. Lastly we apply Corollary \ref{cor:maintauberian} to each Euler product to produce the asymptotic main terms.

\textbf{Remark:} The case $a_{\inv}(\mathcal{L})=\infty$ can only occur if $L_p = H^1_{ur}(K_p,T)$ for all $p\not\in S$. This implies we are in the case of Wiles' original theorem \ref{thm:Wiles}, so that $H^1_{\mathcal{L}}(K,T)$ is necessarily finite. Indeed, by construction $b_{\inv}(K,T)=1$ in this case so that Theorem \ref{thm:asymptoticWiles} gives
\[
|H^1_{\mathcal{L}}(K,T;X)| \sim c_{\inv}(\mathcal{L}) = |H^1_{\mathcal{L}}(K,T)|\,.
\]
Thus, for the remainder of the proof we may assume that $a_{\inv}(\mathcal{L}) < \infty$.

\subsubsection{Reduction to finite order Selmer groups}\label{subsec:finordSelmer}

We recall condition (c) for the admissible ordering:
\begin{itemize}
\item[(c)]{For all places $p$, if $f,f'\in H^1(K_p,T)$ such that $\langle f|_{I_p}\rangle = \langle f'|_{I_p}\rangle \le H^1(I_p,T)$ then $\nu_p(\inv(f))=\nu_p(\inv(f'))$.
}
\end{itemize}
This tells us that the ordering doesn't really see individual elements or the splitting type. We make the following definition to describe what the ordering is really seeing:

\begin{definition}\label{def:Lambda}
Define the poset
\begin{align*}
\Lambda(K,T) &= \left\{\lambda = (\lambda_p)_{p\nmid\infty} :  \substack{\ds \lambda_p\le H^1(I_p,T)\text{ cyclic},\\ \ds\lambda_p=1\text{ for all but finitely many }p}\right\},
\end{align*}
which is ordered by inclusion $\lambda \le \lambda'$ if and only if $\lambda_p\le \lambda'_p$ for all $p$.
\begin{enumerate}
\item[(a)]{
There is a natural map $H^1(K,T)\rightarrow \Lambda(K,T)$ given by $f\mapsto \lambda(f)$ where $\lambda(f)_p=\langle f|_{I_p}\rangle\le H^1(K_p,T)$ for each $p\nmid \infty$.
}
\item[(b)]{
Define $\Lambda(\mathcal{L})$ to be the subposet satisfying the local conditions $\mathcal{L}=(L_p)$, i.e.
\[
\Lambda(\mathcal{L}) = \left\{\lambda\in \Lambda(K,T) : \lambda_p \le \res_{I_p}(L_p) \text{ for each }p\right\}.
\]
Clearly $\Lambda(\mathcal{L})$ is closed under containment, i.e. $\lambda\in \Lambda(\mathcal{L})$ and $\lambda'\le \lambda$ implies $\lambda'\in\Lambda(\mathcal{L})$.
}
\end{enumerate}
\end{definition}

We will now convert the information in $H_{\mathcal{L}}(s)$ and Theorem \ref{thm:asymptoticWiles} from individual elements $f\in H^1_{\mathcal{L}}(K,T)$ to cyclic subgroups $\lambda\in \Lambda(\mathcal{L})$.

\begin{lemma}\label{lem:Lambda}
Let $\Lambda(K,T)$ be as above and suppose $\mathcal{L}$ and $\inv$ satisfy the hypotheses of Theorem \ref{thm:asymptoticWiles}. Then the following hold:
\begin{itemize}
\item[(i)]{The induced ordering $\inv:\Lambda(K,T)\rightarrow I_K$ given by
\begin{align*}
\inv(\lambda) &= \prod_{p} p^{\nu_p(\inv(f_p))}
\end{align*}
where $f_p$ is any choice of generator $\lambda_p = \langle f_p|_{I_p}\rangle$ is well-defined, and satisfies $\inv(\lambda(f))=\inv(f)$,
}

\item[(ii)]{We can rewrite the Dirichlet series $H_{\mathcal{L}}(s)$ as
\[
H_{\mathcal{L}}(s) = \sum_{\lambda\in \Lambda(\mathcal{L})} \#\{f\in H^1_{\mathcal{L}}(K,T) : \lambda(f)=\lambda\} \mathcal{N}_{K/\Q}(\inv(\lambda))^{-s}.
\]
}

\item[(iii)]{Define the family of local conditions $\mathcal{L}(\lambda)$ by $L(\lambda)_p=\res_{I_p}^{-1}(\lambda_p) \cap L_p$. $\mathcal{L}(\lambda)$ satisfies the hypotheses of Wiles' Theorem \ref{thm:Wiles}.}

\item[(iv)]{The Selmer group associated to $\mathcal{L}(\lambda)$ is finite and has a partition
\[
H^1_{\mathcal{L}(\lambda)}(K,T) = \coprod_{\lambda'\le \lambda} \{f\in H^1_{\mathcal{L}}(K,T):\lambda(f)=\lambda'\}\,,
\]
where we use $\coprod$ to denote a disjoint union.
}
\end{itemize}
\end{lemma}

This lemma is all essentially a consequence of condition (c). Wiles' theorem only applies to families of subgroups $L_p\le H^1(K_p,T)$, not just subsets, so if we want to apply this theorem we will need $L(\lambda)_p$ to actually be a group, which follows from part (iii).

\begin{proof}
\begin{itemize}
\item[(i)]{
Condition (c) implies that $\inv$ is well-defined on $\Lambda$, as $\nu_p(\inv(f_p))$ is independent of the choice of generator for $\langle f_p|_{I_p}\rangle$. By definition, $\lambda(f)_p = \langle f|_{I_p}\rangle$ so that $\nu_p(\inv(\lambda))=\nu_p(\inv(\res_p(f))) = \nu_p(\inv(f))$.
}

\item[(ii)]{
This follows immediately from the fact that $\inv(\lambda(f))=\inv(f)$ and that $f\in H^1_{\mathcal{L}}(K,T)$ implies $\lambda(f)\in \Lambda(\mathcal{L})$.
}

\item[(iii)]{
$L(\lambda)_p=\res_{I_p}^{-1}(\lambda_p) \cap L_p$ is an intersection of groups, so it is itself a group. It suffices to show that $L(\lambda)_p=H^1_{ur}(K_p,T)$ for all but finitely many places $p$. One of the assumptions in Theorem \ref{thm:asymptoticWiles} is that $H^1_{ur}(K_\sigma,T) \subset L_\sigma$ for all $\sigma\in \Gal(F/K)$, i.e. for all $p\not\in S$, $H^1_{ur}(K_p,T) \subset L_p$. $S$ is finite, and $\lambda_p=0$ for all but finitely many places. This implies that for all but finitely many places
\begin{align*}
L(\lambda)_p &= \res_{I_p}(0) \cap L_p\\
&= H^1_{ur}(K_p,T) \cap L_p\\
&= H^1_{ur}(K_p,T).
\end{align*}}

\item[(iv)]{
The union is necessarily disjoint, as $\lambda(f)$ cannot be more than one element of $\Lambda(K,T)$ at the same time. Suppose $f\in H^1_{\mathcal{L}(\lambda)}(K,T)$, so that $\res_p(f)\in \res_{I_p}^{-1}(\lambda_p) \cap L_p$ for all places $p$. In particular, this implies $\langle f|_{I_p}\rangle \le \lambda_p$, so that $\lambda(f)\le \lambda$. This shows one inclusion
\[
H^1_{\mathcal{L}(\lambda)}(K,T) \subset \coprod_{\lambda'\le \lambda} \{f\in H^1_{\mathcal{L}}(K,T):\lambda(f)=\lambda'\}.
\]
For the reverse direction, if $\lambda(f)\le \lambda$ and $f\in H^1_{\mathcal{L}}(K,T)$ this implies $f|_{I_p}\in \lambda_p$ and $\res_p(f)\in L_p$ for all places $p$. Putting these together gives $\res_p(f) \in \res_{I_p}^{-1}(\lambda_p)\cap L_p$, so that $f\in H^1_{\mathcal{L}(\lambda)}(K,T)$.
}
\end{itemize}
\end{proof}

This suggests that we can apply a M\"obius inversion to write $H_{\mathcal{L}}(s)$ in terms of the finite order Selmer groups $H^1_{\mathcal{L}(\lambda)}(K,T)$. In particular, Lemma \ref{lem:Lambda}(iv) looks like the setup for M\"obius inversion.

\begin{definition}
Given a poset $\Lambda$, the \text{M\"obius function} on the poset is a function $\mu_{\Lambda}:\Lambda\times \Lambda \rightarrow \C$ defined by
\begin{align*}
\mu_{\Lambda}(\lambda,\lambda)&=1\\
\mu_{\Lambda}(\lambda_1,\lambda_2) &= -\sum_{\lambda_1<\lambda\le \lambda_2} \mu_{\Lambda}(\lambda, \lambda_2)
\end{align*}
\end{definition}

M\"obius inversion holds for a general M\"obius function on a poset, so that Lemma \ref{lem:Lambda}(iv) implies
\begin{align*}
\#\{f\in H^1_{\mathcal{L}}(K,T): \lambda(f)= \lambda\} = \sum_{\lambda'\le \lambda} \mu_{\Lambda}(\lambda',\lambda)|H^1_{\mathcal{L}(\lambda')}(K,T)|\,.
\end{align*}

Plugging this information into the Dirichlet series proves the following reduction:

\begin{proposition}
\[
H_{\mathcal{L}}(s) = \sum_{\substack{\lambda'\lambda\in \Lambda(\mathcal{L})\\ \lambda'\le \lambda}}\mu_{\Lambda(\mathcal{L})}(\lambda',\lambda)|H^1_{\mathcal{L}(\lambda')}(K,T)|\mathcal{N}_{K/\Q}(\inv(\lambda))^{-s}
\]
\end{proposition}

Before we move on to factoring this Dirichlet series into an Euler product, it will be useful to note that $\Lambda(\mathcal{L})$ and $\mu_{\Lambda(\mathcal{L})}$ both factor over the finite places:

\begin{lemma}
Define $\Lambda_p = \Lambda_p(K,T) = \{\lambda_p\le H^1(I_p,T) : \text{ cyclic}\}$ be a poset ordered by inclusion, $\Lambda_p(L_p) = \{\lambda_p\in \Lambda_p : \lambda_p\le \res_{I_p}(L_p)\}$, and $\mu_p$ the corresponding M\"obius function. Then
\begin{itemize}
\item[(i)]{$\Lambda(\mathcal{L}) = \bigoplus_p \Lambda_p(L_p)$ is a direct sum of posets,}

\item[(ii)]{$\mu_{\Lambda(\mathcal{L})}(\lambda',\lambda) = \prod_{p} \mu_p(\lambda_p',\lambda_p)$,}

\item[(iii)]{$\mu_p(\lambda_p',\lambda_p) =  \mu(|\lambda_p/\lambda'_p|)$, where $\mu$ is the usual M\"obius function on the integers.}
\end{itemize}
\end{lemma}

\begin{proof}
The factorizations follow immediately from tracing through the definition of a direct sum of posets, so that is suffices to show that the M\"obius function on $\Lambda_p(L_p)$ is given by $\mu_p(\lambda_p',\lambda_p) = \mu(|\lambda_p/\lambda_p'|)$. We remark that all elements of $\Lambda_p(L_p)$ are cyclic groups, whose subgroup structure is well-known. We then prove the following via induction on the size of $|\lambda_p/\lambda_p'|$:
\begin{itemize}
\item{If $\lambda'_p\not\le\lambda_p$, then by definition the M\"obius function is equal to zero.}

\item{If $|\lambda_p/\lambda_p'|=1$ then $\lambda_p=\lambda_p'$ so by definition
\[
\mu_p(\lambda_p,\lambda_p) = 1 = \mu(1)\,.
\]
}

\item{
If $\lambda_p' < \lambda_p$ then
\[
\mu_p(\lambda_p',\lambda_p) = -\sum_{\lambda_p'<\lambda \le \lambda_p} \mu(|\lambda_p/\lambda|)\,,
\]
noting that $\lambda_p'<\lambda$ implies $|\lambda_p/\lambda|<|\lambda_p/\lambda_p'|$ so we know each term by the inductive hypothesis. A cyclic group has exactly one subgroup with cardinality $d$ for each $d$ dividing the size of the group. Therefore we can rewrite the summation to be over the sizes of proper subgroups $H< \lambda_p/\lambda_p'$, i.e. over proper divisors of $|\lambda_p/\lambda_p'|$:
\begin{align*}
\mu_p(\lambda_p',\lambda_p) &= -\sum_{\substack{d\mid |\lambda_p/\lambda_p'|\\ d\ne |\lambda_p/\lambda_p'|}} \mu(d)\\
&= \mu(|\lambda_p/\lambda_p'|)\,.
\end{align*}
}
\end{itemize}
\end{proof}

\subsubsection{A finite sum of Euler products}

This section will be about proving the following decomposition:

\begin{proposition}\label{prop:Eulerproduct}
Let $\mathcal{L}$ and $\inv$ be as in Theorem \ref{thm:asymptoticWiles}. Then
\begin{align*}
H_{\mathcal{L}}(s) &= \frac{|H^0(K,T)|}{|H^0(K,T^*)|} \sum_{h\in H^1_{\mathcal{L}(0)^*}(K,T^*)} \prod_{p} Q_p(h,s),
\end{align*}
where the Euler factors can be expressed as
\begin{align*}
Q(h,s) &= \frac{1}{|H^0(K_p,T)|}\sum_{f_p\in L_p} c(h,f_p) \mathcal{N}_{K/\Q}(p)^{-\nu_p(\inv(f_p))s}
\end{align*}
where we let $\Phi(G)$ denote the Frattini subgroup of $G$ and
\begin{align*}
c(h,f_p) =&  \mu\left(\frac{|\langle f_p|_{I_p}\rangle|}{|\langle f_p|_{I_p}\rangle\cap\res_{I_p}(\langle \res_p(h)\rangle^*)|}\right)\\
&\cdot \frac{\#\{g_p\in L_p : \langle g_p|_{I_p}\rangle \Phi(\langle f_p|_{I_p}\rangle) = \langle f_p|_{I_p}\rangle\cap \res_{I_p}(\langle \res_p(h)\rangle^*)\}}{\#\{g_p\in L_p : \langle g_p|_{I_p}\rangle = \langle f_p|_{I_p}\rangle\}}\,.
\end{align*}
In particular,
\begin{enumerate}
\item[(i)]{
for all $h\in H^1_{\mathcal{L}(0)^*}(K,T^*)$ and $f_p\in L_p$, $|c(h,f_p)|\le 1$,
}

\item[(ii)]{
for all $f_p\in L_p$, $c(0,f_p)=1$.
}
\end{enumerate}
\end{proposition}

We remark that $S$ is finite, and for any $p\not \in S$ we assumed $H^1_{ur}(K_p,T) = \res_{I_p}^{-1}(0)\subset L_p$. This implies $L(0)_p = H^1_{ur}(K_p,T)$ for all $p\not\in S$ and $\mathcal{L}(0)$ necessarily satisfies the hypotheses of Wiles' Theorem \ref{thm:Wiles}, implying the dual Selmer group must be finite. This means the summation really is a finite sum.

\begin{proof}
In the previous section we broke down $H_{\mathcal{L}}(s)$ into a sum of finite order Selmer groups. We can apply Wiles' Theorem to describe the size of these Selmer groups:
\[
|H^1_{\mathcal{L}(\lambda')}(K,T)| = |H^1_{\mathcal{L}(\lambda')^*}(K,T^*)|\frac{|H^0(K,T)|}{|H^0(K,T^*)|} \prod_p \frac{|L(\lambda')_p|}{|H^0(K,T)|}.
\]
We also showed $\Lambda(K,T)$ is a direct sum of local posets and computed its M\"obius function. We can put this all together to find that
\[
H_{\mathcal{L}}(s) = \frac{|H^0(K,T)|}{|H^0(K,T^*)|}\sum_{\substack{\lambda',\lambda\in \Lambda(\mathcal{L})\\\lambda'\le \lambda}}|H^1_{\mathcal{L}(\lambda')^*}(K,T^*)|\prod_{p} \mu(|\lambda_p/\lambda_p'|) \frac{|L(\lambda'_p)_p|}{|H^0(K_p,T)|}\mathcal{N}_{K/\Q}(p)^{-\nu_p(\inv(\lambda))s}.
\]
This almost has an Euler product decomposition. If the dual Selmer group term was not there, we could factor the Dirichlet series immediately. We will prove a lemma that shows the dual Selmer group doesn't affect this strategy too much.

\begin{lemma}
Define the characteristic function of a proposition $P$ by
\[
\mathbf{1}(P) = \begin{cases}
1 & P \text{ true}\\
0 & P\text{ false}.
\end{cases}
\]
Then
\[
|H^1_{\mathcal{L}(\lambda')^*}(K,T^*)| = \sum_{h\in H^1_{\mathcal{L}(0)^*}(K,T^*)} \prod_p \mathbf{1}\left(\res_p(h)\in L(\lambda')_p^*\right)
\]
\end{lemma}

\begin{proof}
For all places $p$ we have by construction $L(0)_p\subset L(\lambda')_p$ for all $\lambda'$ and all $p$, so that $L(\lambda')_p^*\subset L(0)_p^*$ for all $\lambda'$ and all $p$. This implies
\[
H^1_{\mathcal{L}(\lambda')^*}(K,T^*) \subset H^1_{\mathcal{L}(0)^*}(K,T^*),
\]
which is finite by Wiles' theorem. We can write
\[
H^1_{\mathcal{L}(\lambda')^*}(K,T^*) = \{f\in H^1_{\mathcal{L}(0)^*}(K,T^*) : \forall p,\res_p(f)\in L(\lambda')_p^*\}.
\]
We only care about the order of this set, so we can write it as a sum of characteristic functions
\begin{align*}
|H^1_{\mathcal{L}(\lambda')^*}(K,T^*)| &= \sum_{h\in H^1_{\mathcal{L}(0)^*}(K,T^*)} \mathbf{1}\left(\forall p, \res_p(h)\in L(\lambda')_p^*\right)\\
&= \sum_{h\in H^1_{\mathcal{L}(0)^*}(K,T^*)} \prod_p \mathbf{1}\left(\res_p(f)\in L(\lambda_p')_p^*\right).
\end{align*}
\end{proof}

We can move the (finite) sum over $H^1_{\mathcal{L}(0)^*}(K,T^*)$ to the outside, so that the Dirichlet series is given by
\begin{align*}
&\sum_{h\in H^1_{\mathcal{L}(0)^*}(K,T^*)} \sum_{\substack{\lambda'\lambda\in \Lambda(\mathcal{L})\\ \lambda'\le \lambda}} \prod_p \mu(|\lambda_p/\lambda_p'|)\mathbf{1}\left(\res_p(h)\in L(\lambda_p')_p^*\right)\left(\frac{|L(\lambda'_p)_p|}{|H^0(K_p,T)|} \mathcal{N}_{K/\Q}(p)^{-v_p(\inv(\lambda_p))s}\right)\\
=&\sum_{h\in H^1_{\mathcal{L}(0)^*}(K,T^*)} \prod_p \left(\sum_{\substack{\lambda'_p,\lambda_p\in \Lambda_p(L_p)\\ \lambda'_p\le \lambda_p}}\mu(|\lambda_p/\lambda_p'|)\mathbf{1}\left(\res_p(h)\in L(\lambda_p')_p^*\right)\frac{|L(\lambda_p')_p|}{|H^0(K_p,T)|} \mathcal{N}_{K/\Q}(p)^{-\nu_p(\inv(\lambda_p))s}\right)
\end{align*}

All that remains is to undo the M\"obius inversion on each local factor in order to simplify the Euler factors. For fixed $\lambda_p$, we can express the sum over $\lambda_p'$ as
\begin{align*}
\sum_{\lambda_p'\le \lambda_p} \mu(|\lambda_p/\lambda_p'|)\mathbf{1}\left(\res_p(h)\in L(\lambda_p')_p^*\right)|L(\lambda_p')_p|&= \sum_{\lambda_p' \le \lambda_p \cap \res_{I_p}(\langle\res_p(h)\rangle^{*})} \mu(|\lambda_p/\lambda_p'|)|L(\lambda_p')_p|\,.
\end{align*}
We use the multiplicativity of the M\"obius function via the formula
\[
\mu(ab) = \mu(a)\mu(b) \mathbf{1}\left(\mu(ab)\ne 0\right)
\]
to rewrite
\begin{align*}
\sum_{\lambda_p' \le \lambda_p \cap \res_{I_p}(\langle\res_p(h)\rangle^{*})} \mu(|\lambda_p/\lambda_p'|)|L(\lambda_p')_p|
\end{align*}
as
\begin{align*}
\mu\left(\frac{|\lambda_p|}{|\lambda_p\cap\res_{I_p}(\langle \res_p(h)\rangle^*)|}\right)\sum_{\lambda_p' \le \lambda_p \cap \langle\res_p(h)\rangle^{*}} \mu\left(\frac{|\lambda_p\cap \res_{I_p}(\langle \res_p(h)\rangle^*)|}{|\lambda_p'|}\right)\mathbf{1}\left(\mu\left(|\lambda_p/\lambda_p'|\right)\ne 0\right)|L(\lambda_p')_p|\,.
\end{align*}
$\lambda_p$ is a cyclic group, which implies
\begin{align*}
\mu\left(|\lambda_p/\lambda_p'|\right)=0 &&\text{if and only if}&& \Phi(\lambda_p)\not{\le}\lambda_p'\,,
\end{align*}
where $\Phi(\lambda_p)$ is the Frattini subgroup of $\lambda_p$. Thus we can rewrite the summation as
\begin{align*}
\mu\left(\frac{|\lambda_p|}{|\lambda_p\cap\res_{I_p}(\langle \res_p(h)\rangle^*)|}\right)\sum_{\Phi(\lambda_p)\le \lambda_p' \le \lambda_p \cap \res_{I_p}(\langle\res_p(h)\rangle^{*})} \mu\left(\frac{|\lambda_p\cap \res_{I_p}(\langle \res_p(h)\rangle^*)|}{|\lambda_p'|}\right)|L(\lambda_p')_p|\,.
\end{align*}
We remark that the M\"obius function out front is equal to zero if an only if $\Phi(\lambda_p)\not\le \lambda_p\cap\res_{I_p}(\langle\res_p(h)\rangle^*)$, so that the summation is empty exactly when the M\"obius function makes the whole expression equal to zero anyways. This is a M\"obius inversion, from which we get
\[
\mu\left(\frac{|\lambda_p|}{|\lambda_p\cap\res_{I_p}(\langle \res_p(h)\rangle^*)|}\right)\#\{f_p\in L_p : \langle f_p|_{I_p}\rangle \Phi(\lambda_p) = \lambda_p\cap \langle \res_{I_p}(\res_p(h)\rangle^*)\}\,.
\]
This implies
\begin{align*}
Q(h,s) =& \frac{1}{|H^0(K_p,T)|}\sum_{\lambda_p\in \Lambda_p(L_p)}\mu\left(\frac{|\lambda_p|}{|\lambda_p\cap\res_{I_p}(\langle \res_p(h)\rangle^*)|}\right)\\
&\cdot \#\{f_p\in L_p : \langle f_p|_{I_p}\rangle \Phi(\lambda_p) = \lambda_p\cap \langle \res_{I_p}(\res_p(h)\rangle^*)\}\mathcal{N}_{K/\Q}(p)^{-\nu_p(\inv(\lambda_p))s}\\
&= \frac{1}{|H^0(K_p,T)|}\sum_{f_p\in L_p} c(h,f_p) \mathcal{N}_{K/\Q}(p)^{-\nu_p(\inv(f_p))s}\,.
\end{align*}
To conclude the proof, we remark that for any cyclic group $G$ and $H\le G$, then
\[
\#\{g\in G : \langle g\rangle \Phi(G) = H\} = \begin{cases}
\phi(|H|) & \Phi(G)\le H\\
0 & \Phi(G)\not\le H\,,
\end{cases}
\]
where $\phi$ is the Euler totient function. Thus
\begin{align*}
&\frac{\#\{g_p\in L_p : \langle g_p|_{I_p}\rangle \Phi(\langle f_p|_{I_p}\rangle) = \langle f_p|_{I_p}\rangle\cap \res_{I_p}(\langle \res_p(h)\rangle^*)\}}{\#\{g_p\in L_p : \langle g_p|_{I_p}\rangle = \langle f_p|_{I_p}\rangle\}}\\
&\le \frac{|H^1_{ur}(K_p,T)\cap L_p| \cdot \phi(|\langle f_p|_{I_p}\rangle \cap\res_{I_p}(\langle \res_p(h)^*\rangle)|)}{|H^1_{ur}(K_p,T)\cap L_p| \cdot \phi(|\langle f_p|_{I_p}\rangle|)}\,,
\end{align*}
with equality if and only if $\Phi(\langle f_p|_{I_p}\rangle)\le \res_{I_p}(\langle \res_p(h)^*\rangle)$. In particular, $\phi(d)\le \phi(n)$ for any divisor $d\mid n$ implies $|c(h,f_p)|\le 1$. Clearly $\res_p(0)^*=H^1(K_p,T)$, so that evaluating gives $c(0,f_p)=1$.
\end{proof}

\subsubsection{Applying a Tauberian Theorem}

The goal of this section will be to apply Corollary \ref{cor:maintauberian} to the series
\[
Q(h,s) = \prod_p Q_p(h,s)
\]
appearing in Proposition \ref{prop:Eulerproduct}. 

\begin{proposition}\label{prop:Q}
Let $\mathcal{L}$ and $\inv$ be as in Theorem \ref{thm:asymptoticWiles}, and set $a=a_{\inv}(\mathcal{L})$ and $b=b_{\inv}(\mathcal{L})$. Then there exist real constants $c(h,K,T)$ such that
\begin{align*}
|H^1_{\mathcal{L}}(K,T;X)| = \left(\sum_{h\in H^1_{\mathcal{L}(0)^*}(K,T^*)} c(h,K,T)\right)X^{1/a}(\log X)^{b-1} + o\left(X^{1/a}(\log X)^{b-1}\right)
\end{align*}
as $X\to \infty$.

Moreover, the following hold:
\begin{enumerate}
\item[(i)]{$c(0,K,T)>0$,}
\item[(ii)]{If for all $p$, $\res_{I_p}(L_p^{[a]})\not\le \res_{I_p}(\langle \res_p(h)\rangle^*)$, then $c(h,K,T)=0$.}
\end{enumerate}
\end{proposition}

\begin{proof}
If $h\in H^1(K,T^*)$, we will show that $\res_p(h)$ is Frobenian. Indeed, if we choose some $\widetilde{h}\in Z^1(K,T^*)$ representing $h$, then the Galois correspondence in Lemma \ref{lem:Z1} implies that $\widetilde{h}*(\phi*\chi):G_K\rightarrow T\rtimes \Aut(T)$ is a homomorphism, where $\phi:G_K\rightarrow \Aut(T)$ is the Galois action and $\phi*\chi$ is the Galois action on $T^*=\Hom(T,\mu)$ (i.e. the original action twisted by $\chi$). Let $F/K$ contain the fields of definition of $\widetilde{h}*(\phi*\chi)$, $T$, and $T^*$. Any place $p$ unramified in all three actions has $(\widetilde{h}*\phi*\chi)(\Fr_p)$ and $(\phi*\chi)(\Fr_p)$ both determined by $\Leg{F/K}{p}$. This implies $\res_p(\widetilde{h})$ is Frobenian in $F/K$, and thus so is $\res_p(h)$.

There are only finitely many $h\in H^1_{\mathcal{L}(0)^*}(K,T)$, so let $F/K$ be a finite extension containing the fields of definition of $T$ and $T^*$, and for which $\mathcal{L}$, $\inv$, and $p\mapsto \res_p(h)$ for each $h\in H^1_{\mathcal{L}(0)^*}(K,T)$ are Frobenian. Denote
\[
Q_p(h,x) = \frac{1}{|H^0(K_p,T)|}\sum_{f_p\in L_p} c(h,f_p) x^{\nu_p(\inv(f_p))}\,,
\]
where $c(h,f_p)$ is as in Proposition \ref{prop:Eulerproduct}. This, together with Lemma \ref{lem:KC}, implies that $p\mapsto Q_p(h,x)$ is Frobenian in $F/K$. Thus $Q(h,s)$ satisfies the hypotheses of Corollary \ref{cor:maintauberian} for each $h$, and so contributes an asymptotic term
\[
c(Q(h,s)) X^{1/a(Q(h,s))}(\log X)^{b(Q(h,s))}\,.
\]

$Q(0,s)$ is a series of all positive coefficients for which
\begin{align*}
a(Q(0,s)) = \min_{\sigma\in \Gal(F/K)} \min_{\substack{f_\sigma\in L_\sigma\\ f_\sigma\not\in H^1_{ur}(K_\sigma,T)}}\nu_\sigma(\inv(f_\sigma)) = a_\inv(\mathcal{L}) = a
\end{align*}
and
\begin{align*}
b(Q(0,s)) = \frac{1}{[F:K]}\sum_{\sigma\in\Gal(F/K)} \frac{|L_p^{[a]}|}{|H^0(K_\sigma,T)|} = b_\inv(\mathcal{L}) = b\,.
\end{align*}
Thus the series $Q(0,s)$ contributes the asymptotic term
\[
c(0,K,T) X^{1/a}(\log X)^{b-1}\,,
\]
where
\[
c(0,K,T) = \frac{G(1)\Res_{s=1}\zeta_K(s)^b}{a^b\Gamma(b)} > 0
\]
as in the proof of Theorem \ref{thm:MBlocal}.

For $Q(h,s)$ with $h\ne 0$, clearly $a(Q(h,s))\ge a(Q(0,h))=a$. If $a(Q(h,s))=a$, then
\begin{align*}
b(Q(h,s)) &= \frac{1}{[F:K]}\sum_{\sigma\in \Gal(F/K)} \frac{1}{|H^0(K_\sigma,T)|}\sum_{f_\sigma\in L_\sigma[a]} c(h,f_\sigma)\,.
\end{align*}
If $f_p|_{I_p}\in \res_{I_p}(\langle \res_p(h)\rangle^*)$, then
\begin{align*}
c(h,f_p) &= 1 \cdot \frac{\#\{g_p\in L_p : \langle g_p|_{I_p}\rangle \Phi(\langle f_p|_{I_p}\rangle) = \langle f_p|_{I_p}\rangle\}}{\#\{g_p\in L_p : \langle g_p |_{I_p}\rangle = \langle f_p|_{I_p}\rangle\}}\\
&= 1
\end{align*}
by the Frattini subgroup being the subgroup of nongenerators. Otherwise, we can bound
\begin{align*}
|c(h,f_p)|&\le \frac{\#\{\text{generators of }\langle f_p|_{I_p}\rangle \cap \res_{I_p}(\langle \res_p(h)\rangle^*)\}}{\#\{\text{generators of }\langle f_p|_{I_p}\rangle\}}\,.
\end{align*}
The number of generators of a cyclic group of order $n$ is exactly $\phi(n)$ for $\phi$ the Euler $\phi$-function. The number of generators of a proper cyclic subgroup of order $d\mid n$ is then $\phi(d)\le \phi(n)$, which is strict if and only if $d\ne n$, which implies $|c(h,f_p)|< 1$. Thus
\begin{align*}
b(Q(h,s)) &\le \frac{1}{[F:K]}\sum_{\sigma\in \Gal(F/K)} \frac{1}{|H^0(K_\sigma,T)|}\sum_{f_\sigma\in L_\sigma^{[a]}} 1\\
&=b\,,
\end{align*}
with equality if and only if $f_p|_{I_p}\in \res_{I_p}(\langle\res_p(h)\rangle^*)$ for each $p\not\in S$ and each $f_p\in L_p^{[a]}$. This contributes an asymptotic term
\[
c(h,K,T)X^{1/a}(\log X)^{b-1} + o(X^{1/a}(\log X)^{b-1})\,,
\]
where $c(h,K,T)=0$ if $\res_{I_p}(L_p^{[a]})\not\le \res_{I_p}(\langle \res_p(h)\rangle^*)$ or $a(Q(h,s))>a$. This concludes the proof.
\end{proof}

The asymptotic Wiles theorem then follows from showing that the sum of $c(h,K,T)$ is necessarily positive.

\begin{proof}[Proof of Theorem \ref{thm:asymptoticWiles}]
Let $T' = \langle f_p(I_p) : f_p\in L_p^{[a]}\rangle\le T$. In particular, if we define $\mathcal{L}(T')=(L(T')_p)$ by $L(T')_p = L_p\cap H^1(K_p,T')$ then
\[
H^1_{\mathcal{L}(T')}(K,T';X)\subset H^1_{\mathcal{L}}(K,T;X)\,,
\]
and both $a_{\inv}(\mathcal{L}) = a_{\inv}(\mathcal{L}(T'))$ and $b_{\inv}(\mathcal{L})= b_{\inv}(\mathcal{L}(T'))$ by $H^1(K_p,T')$ containing all the elements $f_p\in L_p$ with $\nu_p(\inv(f_p))$ minimal by construction.

$T'$ is the minimal subgroup with this property, given any nonzero $h\in H^1_{\mathcal{L}(T')(0)^*}(K,(T')^*)$, there exists at least one $\sigma\in \Gal(F/K)$ such that $\res_\sigma(h) \ne 0$, i.e. for any place $p\not\in S$ with $\Leg{F/K}{p}$ conjugate to $\sigma$, $\res_p(h)\ne 0$. All places ramified in $h$ belong to $S$, so $h$ is necessarily unramified at $p$. Taking duals, this implies $\langle \res_\sigma(h)\rangle^* \ne H^1(K_p,T')$ is not everything and $H^1_{ur}(K_p,T') \subset \langle \res_p(h)\rangle^*$ by $\langle \res_p(h)\rangle \le H^1_{ur}(K_p,(T')^*)^*$ and $H^1_{ur}(K_p,(T')^*)^* = H^1_{ur}(K_p,T')$. This implies there is at least one $f_p\in H^1(K_p,T')$ for which $f_p\not\in \langle\res_p(h)\rangle^*$, so necessarily $f_p H^1_{ur}(K_p,T')\cap \langle \res_p(h)\rangle^* = \emptyset$. Thus $f_p|_{I_p}\not\in \res_{I_p}(\langle\res_p(h)\rangle^*)$. By construction, $H^1(I_p,T')$ is generated by $L_p^{[a]}$, which implies by linearity of the Tate pairing that there exists some $f_p\in L_p^{[a]}$ for which $f_p|_{I_p}\not\in \res_{I_p}(\langle \res_p(h)\rangle^*)$. By Proposition \ref{prop:Q}(ii), this implies $c(h,K,T')=0$. Therefore
\begin{align*}
|H^1_{\mathcal{L}}(K,T;X)| \ge |H^1_{\mathcal{L}(T')}(K,T';X)| \sim c(0,K,T')X^{1/a}(\log X)^{b-1}\,.
\end{align*}
Proposition \ref{prop:Q}(i) implies $c(0,K,T')>0$, concluding the proof.
\end{proof}

\newpage

\section{Applications to Number Field Counting}\label{sec:counting}

In Lemma \ref{lem:KC} we proved that the family $(H^1(K_p,T(\pi)))_p$ is Frobenian in $F/K$ with $F$ containing the fixed field of $\ker\pi$ and $\mu_{|T|}$. Conditions (a) and (b) of Theorem \ref{thm:asymptoticWiles} are also trivially satisfied:

\begin{itemize}
\item[(a)]{
This is not really a condition, it is just labeling $S$.
}

\item[(b)]{
This is trivial, as $H^1_{ur}(K_\sigma,T(\pi))\subset H^1(K_\sigma,T(\pi)) = H^1(K_p,T(\pi))$ for any regular place with $\Leg{F/K}{p}$ conjugate to $\sigma$.
}
\end{itemize}

In order to apply Theorem \ref{thm:asymptoticWiles} we only need to check that the $\pi$-discriminant is Frobenian and satisfies condition (c). Unfortunately, it turns out that $\disc_\pi$ need not satisfy condition (c) but instead satisfies something slightly weaker.

\begin{lemma}\label{lem:pidiscFrob}
The $\pi$-discriminant is Frobenian in $F/K=L(\mu_{|T|})/K$, where $L = \overline{K}^{\ker \pi}$ is the fixed field of $\pi$. Moreover, there exists admissible invariants $\disc_\pi^{\uparrow}$ and $\disc_\pi^{\downarrow}$ satisfying the hypotheses of Theorem \ref{thm:asymptoticWiles} such that
\begin{itemize}
\item{
for all $p\not\in S$ and $f\in H^1(K,T(\pi))$,
\[
\nu_p(\disc_\pi^{\downarrow}(f))=\nu_p(\disc_p(f))=\nu_p(\disc_\pi^{\uparrow}(f))
\]
}
\item{
for all $f\in H^1(K,T(\pi))$
\[
\mathcal{N}_{K/\Q}(\disc_\pi^{\downarrow}(f)) \le \mathcal{N}_{K/\Q}(\disc_\pi(f)) \le \mathcal{N}_{K/\Q}(\disc_\pi^{\uparrow}(f)).
\]
}
\end{itemize}
\end{lemma}

It is the fact that $\disc_\pi$ does not satisfy condition (c) of the Asymptotic Wiles' Theorem that results in Theorm \ref{thm:MalleZ1} only achieving the asymptotic growth rate instead of the main term on the nose (i.e. why the result has $\asymp$ instead of $\sim$). Bounding $\disc_\pi$ above and below by admissible orderings satisfying the hypotheses of Theorem \ref{thm:asymptoticWiles} produces upper and lower bounds for the counting function, and the fact that the bounds agree with $\disc_\pi$ implies that the upper and lower bounds for the counting function have the same order of magnitude.

\begin{proof}[Proof of Lemma \ref{lem:pidiscFrob}]
We first check that $\disc_\pi$ is an admissible ordering. Indeed, the $\pi$-discriminant is defined on crossed homomorphisms and Lemma \ref{lem:pidisc} implies it factors through the quotient by the coboundary relation, so it suffices to prove this condition for crossed homomorphisms. By the definition, $\nu_p(\disc(f*\pi))$ depends only on $(f*\pi)|_{I_p}$, and $\pi$ being fixed implies that $\nu_p(\disc_\pi(f))$ depends only on $f|_{I_p}$. Moreover, for all places $p$ unramified in $\pi$, $\nu_p(\disc(f*\pi))$ depends only on $(f*\pi)|_{I_p} = f|_{I_p}$ and equals zero exactly when $f|_{I_p}=0$.

Next ,we show that $\disc_\pi$ is Frobenian:
\begin{itemize}
\item[(a)]{
By construction, $L$ is the field of definition of $T$. The field of definition of $T^*=\Hom(T,\mu_{|T|})$ is certainly contained in $F=L(\mu_{|T|})$, so that $F$ contains the fields of definition of both $T$ and $T^*$.
}

\item[(b)]{
This is not really a condition, just choose $S$ to be exactly the set of places ramified in $F$ together with all places $p\mid |T|\infty$.
}

\item[(c)]{
This follows from the fact that $S$ contains all primes ramified in $\pi$, which we showed above are the only possible primes which could violate this condition.
}

\item[(d)]{We chose $S$ large enough so that $p\nmid |T|\infty$, so in particular $p$ may only be tamely ramified. As above, we get
\[
\nu_p(\disc(f*\pi)) = \ind(f(\tau_p))\,.
\]
Lemma \ref{lem:KC} implies that this is the same as $\ind(f(\tau))$ for the generator $\tau\in G_m$. This depends only on $\sigma=\Leg{F/K}{p}\equiv m \mod |T|$ (up to conjugation), so we may take $\nu_\sigma(\res_p(f))=\ind(f(\tau))$.
}
\end{itemize}
Next we show that condition (c) of Theorem \ref{thm:asymptoticWiles} is satisfied \emph{at all but finitely many places} $p$, rather than all places. Suppose $p$ is not ramified in $\pi$ and $p\nmid |T|\infty$. Then $p$ must be tamely ramified and
\begin{align*}
\nu_p(\disc_\pi(f))&= \nu_p(\disc(f*\pi))\\
&= \ind(f(\tau_p)\pi(\tau_p))\\
&= \ind(f(\tau_p))\\
&= n - \#\{\text{orbits of }f(\tau_p)\}\,.
\end{align*}
The number of orbits of $x\in S_n$ is the same as the number of orbits of $x^N$ for any $N$ coprime to the order of $x$, and $\langle f|_{I_p}\rangle = \langle f'|_{I_p}\rangle$ implies $f(\tau_p) = f'(\tau_p)^N$ for some $N$ coprime to the order of $f(\tau_p)$. Thus $\ind(f(\tau_p))=\ind(f'(\tau_p))$, and so $\nu_p(\disc(f))=\nu_p(\disc(f'))$.

We define $\disc_\pi^{\uparrow}$ and $\disc_\pi^\downarrow$ by
\begin{align*}
\nu_p(\disc_{\pi}^{\downarrow}(f)) &= \begin{cases}
\nu_p(\disc_\pi(f)) & p\not\in S\\
\ds\min_{g\in H^1(K_p,T(\pi))} \nu_p(\disc_\pi(f)) & p\in S
\end{cases}\\
\nu_p(\disc_{\pi}^{\uparrow}(f)) &= \begin{cases}
\nu_p(\disc_\pi(f)) & p\not\in S\\
\ds\max_{g\in H^1(K_p,T(\pi))} \nu_p(\disc_\pi(f)) & p\in S
\end{cases}
\end{align*}
These satisfy the conditions in the statement of Lemma \ref{lem:pidiscFrob} by construction. The order of $p$ dividing these invariants is constant when $p\in S$, and so automatically satisfies condition (c) of Theorem \ref{thm:asymptoticWiles}, and for $p\not\in S$ these invariants inherit this property from $\disc_\pi$. Because $\disc_\pi$, $\disc_\pi^{\uparrow}$, and $\disc_\pi^{\downarrow}$ agree at all but finitely many places, $\disc_\pi$ being admissible and Frobenian implies $\disc_\pi^{\uparrow}$ and $\disc_\pi^{\downarrow}$ are also admissible Frobenian.
\end{proof}

This implies Theorem \ref{thm:MalleZ1} is a special case of Theorem \ref{thm:asymptoticWiles} (up to a nonzero constant multiple per Lemma \ref{lem:pidisc}) with $L_p=H^1(K_p,T(\pi))$ for all $p$ and $\inv=\disc_\pi^{\uparrow}$ or $\disc_\pi^{\downarrow}$ giving lower and upper bounds respectively for $|H^1(K,T(\pi);X)|$, where it now suffices to check that the $a$- and $b$-invariants agree. For convenience, let $A(T) = \{t\in T : \ind(t)=a(T)\}$. This is the set that determines the $a$- and $b$-invariants.

\begin{lemma}
Let $G\subset S_n$ be a transitive subgroup, $T\normal G$ an abelian $\ell$-group, and $\pi:G_K\rightarrow G$ a homomorphism. If $\mathcal{L}$ is defined by the trivial relations $L_p=H^1(K_p,T(\pi))$ for all places $p$ and $\disc_\pi^\downarrow$ and $\disc_\pi^\uparrow$ are as in Lemma \ref{lem:pidiscFrob} then
\[
a_{\disc_{\pi}}(\mathcal{L})=a_{\disc_\pi^\downarrow}(\mathcal{L}) = a_{\disc_\pi^\uparrow}(\mathcal{L}) = a(T)
\]
and
\[
b_{\disc_{\pi}}(\mathcal{L}) = b_{\disc_\pi^\downarrow}(\mathcal{L}) = b_{\disc_\pi^\uparrow}(\mathcal{L}) = b(K,T(\pi)),
\]
the invariants given by the Malle-Bhargava principle.
\end{lemma}

This proof takes advantage of the fact that we can explicitly realize the product of local factors given in the Malle-Bhargava principle as a summand of the Dirichlet series $H_{\mathcal{L}}(s)$.

\begin{proof}
Lemma \ref{lem:pidiscFrob} states that $\nu_p\disc_\pi^{\downarrow}$, $\nu_p\disc_\pi^{\uparrow}$, and $\nu_p\disc_\pi$ agree at all but finitely many places, implying that
\begin{align*}
a_{\disc_\pi^\downarrow}(\mathcal{L}) &= a_{\disc_\pi^\uparrow}(\mathcal{L}) = a_{\disc_\pi}(\mathcal{L})\text{ and }
b_{\disc_\pi^\downarrow}(\mathcal{L}) = b_{\disc_\pi^\uparrow}(\mathcal{L}) = b_{\disc_\pi}(\mathcal{L}).
\end{align*}
Moreover, Proposition \ref{prop:Q} implies that the series
\begin{align*}
Q(0,s) &= \prod_{p} \frac{1}{|H^1(K_p,T(\pi))|}\sum_{f_p\in H^1(K_p,T(\pi)))} \mathcal{N}_{K/\Q}(\disc_\pi^{\downarrow}(f_p))^{-s}
\end{align*}
has its rightmost pole at $s=1/a_{\disc_\pi}(\mathcal{L})$ of order $b_{\disc_\pi}(\mathcal{L})$. However, Proposition \ref{prop:localZ1toH1}, Theorem \ref{thm:MBlocal}, and Lemma \ref{lem:pidiscFrob} imply that $\nu_p\disc_\pi$ and $\nu_p\disc_\pi^{\downarrow}$ agree at all but finitely many places so that
\begin{align*}
Q(0,s) =& \prod_{p\in S} \left( \frac{1}{|H^0(K_p,T(\pi))|}\sum_{f_p\in H^1(K_p,T(\pi))} \mathcal{N}_{K/\Q}(\disc_\pi(f_p))^{-s}\right)^{-1}\\
&\times \prod_{p\in S} \left( \frac{1}{|H^0(K_p,T(\pi))|}\sum_{f_p\in H^1(K_p,T(\pi))} \mathcal{N}_{K/\Q}(\disc_\pi^{\downarrow}(f_p))^{-s}\right)\\
&\times \prod_{p} \frac{1}{|T|}\sum_{f_p\in Z^1(K_p,T(\pi))} \mathcal{N}_{K/\Q}(\disc_\pi(f_p))^{-s}
\end{align*}
is a product of a holomorphic function (the products over $p\in S$) with a meromorphic function whose rightmost pole is at at $s=1/a(T)$ of order $b(K,T(\pi))$. $Q(0,s)$ can have at most one rightmost pole on the real line, so they must agree concluding the proof.
\end{proof}

This proves Theorem \ref{thm:MalleZ1} as a special case of Theorem \ref{thm:asymptoticWiles}, but these results actually prove a little more. We allowed ourselves to restrict local conditions at finitely many irregular places and still obtained the same order of magnitude. We can restrict local conditions at infinitely many local places as well, so long as the splitting behavior is only restricted at finitely many places, and the order of magnitude of the main term will then be given by the $a$- and $b$-invariants in Theorem \ref{thm:asymptoticWiles}.

\subsection{Counting $(T\normal G)$-towers for $T$ abelian}

We can now perform an inclusion-exclusion argument to prove the same asymptotic for $N(L/K,T\normal G;X)$ as stated in Corollary \ref{cor:alpha}. Again, we will realize this as a special case of a more general result:

\begin{theorem}\label{thm:ratio}
Let $G\subset S_n$ be a transitive subgroup, $T\normal G$ an abelian normal subgroup, $(L/K,\iota_B)$ a fixed $B$-extension with $\iota_B:\Gal(L/K)\xrightarrow{\sim} B$, $\pi:G_K\rightarrow G$ a (not necessarily surjective) homomorphism satisfying $\pi(G_K)T=G$, and $\mathcal{L}$ and $\inv$ satisfy the hypotheses of Theorem \ref{thm:asymptoticWiles}. Then the function
\[
F_{\mathcal{L},\inv}(X) = \frac{|\{f\in H^1_{\mathcal{L},\inv}(K,T(\pi);X) : f*\pi \text{ surjective}\}|}{|H^1_{\mathcal{L},\inv}(K,T(\pi);X)|}
\]
is bounded between $[0,1]$ as $X\to\infty$. Moreover, we get the following special cases:
\begin{enumerate}
\item[(i)]{
If $\pi$ is surjective then $\liminf_{X\to\infty} F(X) > 0$.
}
\item[(ii)]{
If
\[
T = \langle f(I_p) : f\in L_p, p\not\in S\rangle
\]
then $\liminf_{X\to \infty} F(X) >0$.
}
\item[(iii)]{
If
\[
T = \langle f(I_p) : f\in L_p^{[a_{\inv}(\mathcal{L})]},p\not\in S\rangle
\]
then $\lim_{X\to\infty} F(X) = 1$.
}
\end{enumerate}
\end{theorem}

Suppose we take the trivial local conditions $L_p=H^1(K_p,T)$. We noted in the proof of Theorem \ref{thm:MBlocal} that for any $t\in T$ and $p\not\in S$ such that $\Fr_p(F/K)=1$ then there exists an $f\in H^1(K_p,T)$ such that $f(\tau_p)=t$. This implies
\[
T = \bigcup_{\substack{p\not\in S\\ f\in L_p}} f(I_p)\,,
\]
satisfying the conditions in part (ii). Again using Lemma \ref{lem:pidiscFrob} to get the bounds $\disc_\pi^{\downarrow}$ and $\disc_\pi^\uparrow$ which satisfy the hypotheses of Theorem \ref{thm:asymptoticWiles}, we realize Corollary \ref{cor:alpha} as a consequence of Theorem \ref{thm:ratio} by
\begin{align*}
N(L/K,T\normal G;X) &= |T/T^G| \cdot |\{f\in H^1_{\mathcal{L},\disc_\pi}(K,T(\pi);X) : f*\pi \text{ surjective}\}|\\
&\ge |T/T^G| \cdot |\{f\in H^1_{\mathcal{L},\disc_\pi^\uparrow}(K,T(\pi);X) : f*\pi \text{ surjective}\}|\\
&\ge (\liminf_{X\to\infty} F(X)) |T/T^G| |H^1_{\disc_\pi^\uparrow}(K,T(\pi);X)|\\
&\gg X^{1/a(T)}(\log X)^{b(K,T(\pi))-1}.
\end{align*}

Corollary \ref{cor:ram} also follows from Theorem \ref{thm:ratio} as $\ram$ satisfies the conditions of Theorem \ref{thm:asymptoticWiles} with $a=1$ and $b=\#\left(T-\{1\}/\pi*\chi^{-1}\right)$, as well as case (iii) of Theorem \ref{thm:ratio}.

\textbf{Remark:} It is reasonable to expect that $\lim_{X\to\infty} F(X)=1$ in all cases, but we do not achieve this result due to the same limitation in Wiles' theorem which requires condition (c) in Theorem \ref{thm:asymptoticWiles}. Indeed, the proof in cases (i) and (ii) rely on using several invariants which only satisfy condition (c) at all but finitely many places, and so are bounded in the same way as Lemma \ref{lem:pidiscFrob}.

\begin{proof}[Proof of Theorem \ref{thm:ratio}]
Theorem \ref{thm:asymptoticWiles} tells us there exists a positive constant $c_{\inv}(\mathcal{L})$ such that
\begin{align*}
|H^1_{\mathcal{L},\inv}(K,T(\pi);X)| &\sim c_{\inv}(\mathcal{L}) X^{1/a_{\inv}(\mathcal{L})}(\log X)^{b_{\inv}(\mathcal{L})-1}.
\end{align*}

\textbf{Part (i):} Enlarge $S$ to contain a finite set of places $p$ such that $\{\pi(\Fr_p): p\in S\} = G$, which exists by Chebotarev density. Let $\mathcal{L}_\pi$ be defined by
\[
(L_\pi)_p = \begin{cases}
0 & p\in S\\
L_p & \text{else}\,.
\end{cases}
\]
Then the map $f\mapsto f*\pi$ gives a bijection
\[
H^1_{\mathcal{L}_\pi}(K,T(\pi)) \leftrightarrow \{ f\in H^1_{\mathcal{L}}(K,T(\pi)) : \forall p\in S,\ f|_{G_{K_p}}=\pi|_{G_{K_p}}\}\,.
\]
Moreover, for any $f\in H^1_{\mathcal{L}_\pi}(K,T(\pi))$ it necessarily follows that
\begin{align*}
G &= \{\pi(\Fr_p): p\in S\}\\
&=\{(f*\pi)(\Fr_p): p\in S\}\\
&\le (f*\pi)(G_K)\,,
\end{align*}
which implies the map $f\mapsto f*\pi$ induces an inclusion
\[
H^1_{\mathcal{L}_\pi}(K,T(\pi)) \hookrightarrow \{f\in H^1_{\mathcal{L},\inv}(K,T(\pi);X) : f*\pi \text{ surjective}\}\,,
\]
which implies
\begin{align*}
\frac{|\{f\in H^1_{\mathcal{L},\inv}(K,T(\pi);X) : f*\pi \text{ surjective}\}|}{|H^1_{\mathcal{L},\inv}(K,T(\pi);X)|} &\ge \frac{|H^1_{\mathcal{L}_\pi,\inv}(K,T(\pi);X)|}{|H^1_{\mathcal{L},\inv}(K,T(\pi);X)|}\,.
\end{align*}
$\mathcal{L}_\pi$ and $\mathcal{L}$ agree at all but finitely many places, which implies that their $a$- and $b$-invariants are the same. In particular this implies
\begin{align*}
\liminf_{X\rightarrow\infty} \frac{|\{f\in H^1_{\mathcal{L},\inv}(K,T(\pi);X) : f*\pi \text{ surjective}\}|}{|H^1_{\mathcal{L},\inv}(K,T(\pi);X)|} &\ge \frac{c_{\inv}(\mathcal{L}_\pi)}{c_{\inv}(\mathcal{L})} >0\,.
\end{align*}

\textbf{Part (ii):} We will prove that (iii) implies (ii), so that it suffices to prove (iii). Consider the invariant given by the product of ramified places outside of $S$,
\[
\ram^S(f) = \prod_{\substack{p\not\in S\\ f(I_p)\ne 1}}p\,. 
\]
This trivially satisfies the hypotheses of Theorem \ref{thm:asymptoticWiles} with $a_{\inv}(\mathcal{L})=1$ and for all $p\not\in S$, $L_p^{[1]}=L_p-H^1_{ur}(K_p,T)$, so that in particular $\res_{I_p}(L_p^{[1]}) = \res_{I_p}(L_p)-\{0\}$. This implies
\[
\langle f(I_p) : f\in L_p, p\not\in S\rangle = \langle f(I_p): f\in L_p^{[1]}, p\not\in S\rangle\,,
\]
so that whenever $(\mathcal{L},\inv)$ falls under case (ii) necessarily $(\mathcal{L},\ram^S)$ falls under case (iii). By assumption, part (iii) implies that there exists at least one $f_G$ such that $f_G*\pi$ is surjective. This implies that there exists a surjective solution to the embedding problem. For $\inv_{f_G}(f)=\inv(f*f_G)$ we define $\inv_{f_G}^\uparrow$ (and similarly $\inv_{f_G}^\downarrow$)
\[
\nu_p(\inv_{f_G}^{\uparrow}(f)) = \begin{cases}
\nu_p(\inv_{f_G}(f)) & p\in S\text{ or }f_G\text{ ramified at }p\\
\max_{g\in H^1(K_p,T(\pi))} \nu_p(\inv_{f_G}(g)) & \text{else}
\end{cases}
\]
as in Lemma \ref{lem:pidiscFrob} which agrees with $\inv_{f_G}$ at all but finitely many places (and thus has the same $a$- and $b$-invariants) and satisfies the conditions of Theorem \ref{thm:asymptoticWiles}. In particular, ordering the Selmer group by $\inv_{f_G}^\uparrow$ or $\inv_{f_G}^\downarrow$ produces the same asymptotic growth rate in Theorem \ref{thm:asymptoticWiles}.

$T$ abelian implies that $T(f_G*\pi) = T(\pi)$ is the same Galois module. This implies that the map $f\mapsto f*f_G$ induces an inclusion
\begin{align*}
\{f\in H^1_{\mathcal{L},\inv_{f_G}}(K,T(f_G*\pi);X) : f*(f_G*\pi) \text{ surjective}\}\hookrightarrow \{f\in H^1_{\mathcal{L},\inv}(K,T(\pi);X) : f*\pi \text{ surjective}\}\,.
\end{align*}
We also remark that $f_G$ is unramified at all but finitely many places, so $\nu_p(\inv(f))=\nu_p(\inv_{f_G}(f))$ for all but finitely many places. By definition, this implies
\begin{align*}
a_{\inv}(\mathcal{L}) &=a_{\inv_{f_G}}(\mathcal{L}) & b_{\inv}(\mathcal{L}) &= b_{\inv_{f_G}}(\mathcal{L})\,.
\end{align*}
Thus applying Theorem \ref{thm:asymptoticWiles} to the bounds $\inv_{f_G}^\downarrow$ and $\inv_{f_G}^\uparrow$ implies $|H^1_{\mathcal{L},\inv_{f_G}}(K,T(f_G*\pi);X)| \asymp |H^1_{\mathcal{L},\inv}(K,T(\pi);X)|$. Therefore
\begin{align*}
\liminf_{X\to\infty} F_{\mathcal{L},\inv}(X) &\ge \liminf_{X\to\infty}\frac{|\{f\in H^1_{\mathcal{L},\inv_{f_G}^\uparrow}(K,T(f_G*\pi);X) : f*(f_G*\pi) \text{ surjective}\}|}{|H^1_{\mathcal{L},\inv_{f_G}}(K,T(\pi);X)|}\\
& \ge \frac{|\{f\in H^1_{\mathcal{L},\inv_{f_G}^{\uparrow}}(K,T(f_G*\pi);X) : f*(f_G*\pi) \text{ surjective}\}|}{|H^1_{\mathcal{L},\inv_{f_G}^\downarrow}(K,T(f_G*\pi);X)|}\,.
\end{align*}
Noting that orderings $\inv_{f_G}^\uparrow$ and $\inv_{f_G}^\downarrow$ give the same order of magnitude in Theorem \ref{thm:asymptoticWiles} this implies
\begin{align*}
\liminf_{X\to\infty} F_{\mathcal{L},\inv}(X) &\gg \liminf_{X\to\infty}  \frac{|\{f\in H^1_{\mathcal{L},\inv_{f_G}^{\uparrow}}(K,T(f_G*\pi);X) : f*(f_G*\pi) \text{ surjective}\}|}{|H^1_{\mathcal{L},\inv_{f_G}^\uparrow}(K,T(f_G*\pi);X)|},
\end{align*}
which is positive by part (i) as we chose $f_G*\pi$ to be surjective.

\textbf{Part (iii):} We note that $f(G_K)\le T$ implies $f*\pi \equiv \pi \mod T$. Thus $\pi(G_K)T=G$ implies $(f*\pi)(G_K)T=G$. We can partition $H^1_{\mathcal{L}}(K,T(\pi))$ based on the image of $f*\pi$, which under the coboundary relation is well-defined up to $T$-conjugacy. For any $T$-conjugacy class $H$ with $HT=G$ suppose that there exists at least one $f_H$ such that $(f_H*\pi)(G_K)\subset H$. Then we claim that the map $f\mapsto f*f_H^{-1}$ induces a bijection
\[
\left\{f\in H^1_{\mathcal{L}}(K,T(\pi)) : (f*\pi)(G_K)\subset H \right\} \leftrightarrow i_*H^1_{\mathcal{L}(H\cap T)}(K,(H\cap T)(\pi))\,,
\]
where $\mathcal{L}(H\cap T)$ is defined by $L(H\cap T)_p=i_*^{-1}(L_p)$ and $i_*$ is the pushforward along the inclusion map $i_*:H\cap T\hookrightarrow T$. In order to check that $(H\cap T)(\pi)$ is well-defined as a Galois module, it suffices to show that $H\cap T\normal G$ as the action factors through conjugation by $G$ and $c_t(H)\cap T = c_t(H\cap T)$ for a different representative of the $T$-conjugacy class $H$. $T\normal G$ implies that $H\cap T\normal H$ for any representative of the $T$-conjugacy class of subgroups $H$, and $T$ abelian implies $H\cap T\normal T$ so that $H\cap T\normal HT=G$. Next, we remark that for any $f$ a crossed homomorphism representative on the left hand side
\[
(f*f_H^{-1})(G_K)\subset T
\]
by definition, and
\begin{align*}
(f*f_H^{-1})(G_K)&\subset (f*\pi*\pi^{-1}*f_H^{-1})(G_K)\\
&\subset (f*\pi)(G_K) (f_H*\pi)(G_K)^{-1}\\
&\subset H\,,
\end{align*}
which implies $(f*f_H^{-1})(G_K)\subset H\cap T$. For the reverse inclusion we remark that the inverse map $f\mapsto f*f_H$ satisfies
\begin{align*}
(f*f_H*\pi)(G_K)&\subset f(G_K) (f_H*\pi)(G_K)\\
&\subset(H\cap T)H\\
&= H\,.
\end{align*}
Modding out by the coboundaries with coefficients in $T$ instead of $H\cap T$ requires including the pushfoward map $i_*$.

For each $H$ fix a choice of $f_H$, so that we define $\inv_H(f)=\inv(f*f_H^{-1})$ (for example, we can choose $f_G=1$ and note that this is a different choice than in part (ii)). Then, noting that $i_*$ has finite kernel, it follows that
\begin{align*}
|\{f\in H^1_{\mathcal{L}}(K,T(\pi);X) : (f*\pi)(G_K)\subset H\}| &\le |H^1_{\mathcal{L}(H\cap T),\inv_H}(K,T(\pi);X)|\\
&\le |H^1_{\mathcal{L}(H\cap T),\inv_H^{\downarrow}}(K,T(\pi);X)|
\end{align*}
where $\inv_H^\downarrow$ is a bound on the ordering define as in Lemma \ref{lem:pidiscFrob} to satisfy the hypotheses of Theorem \ref{thm:asymptoticWiles}.

Suppose $H\le G$ is such that $a_{\inv}(\mathcal{L}(H\cap T)) = a_{\inv}(\mathcal{L})$. Noting that $\inv$, $\inv_H$ and $\inv_H^{\downarrow}$ agree at all but finitely many places (i.e. all the places unramified in $f_H$), we find that their $a$- and $b$-invariants are necessarily the same on both $\mathcal{L}$ and $i_*\mathcal{L}(H\cap T)$. By definition
\begin{align*}
b_{\inv}(\mathcal{L}(H\cap T)) &= \frac{1}{[F:K]}\sum_{\sigma\in \Gal(F/K)}\frac{|i_*L(H\cap T)_\sigma^{[a_{\inv}(\mathcal{L})]}|}{|H^0(K_\sigma,T(\pi))|}\\
&\le \frac{1}{[F:K]}\sum_{\sigma\in \Gal(F/K)}\frac{|L_\sigma^{[a_{\inv}(\mathcal{L})]}|}{|H^0(K_\sigma,T(\pi))|}\\
&=b_{\inv}(\mathcal{L})\,,
\end{align*}
with equality if and only if
\[
|L(H \cap T)_\sigma^{[a_{\inv}(\mathcal{L})]}|=|L_\sigma^{[a_{\inv}(\mathcal{L})]}|
\]
for all $\sigma\subset \Gal(F/K)$. By definition, $i_*L(H\cap T)_\sigma=i_*i_*^{-1}(L_\sigma)\subset L_{\sigma}$ so
\[
i_*L(H\cap T)_\sigma^{[a_{\inv}(\mathcal{L})]}\subset L_\sigma^{[a_{\inv}(\mathcal{L})]}\,.
\]
Suppose these are equal for each $\sigma \in \Gal(F/K)$. This implies that for every $p\not\in S$ and $f\in L_p^{[a_{\inv}(\mathcal{L})]}$ that $f(I_p)\subset H$ up to $T$-conjugacy (as this is true for all elements of $i_*L(H\cap T)_p$ by construction), which implies
\[
\langle f(I_p) : f\in L_p^{[a_{\inv}(\mathcal{L})]}, p\not\in S \rangle \le H\cap T\,.
\]
By assumption for part (iii) this implies $T=H\cap T$, so that $HT=G$ and the second isomorphism theorem implies
\begin{align*}
|G| &= |T|[G:T]\\
&= |T| [HT: T]\\
&= |T| [H: H\cap T]\\
&= |T| \frac{|H|}{|H\cap T|}\\
& = |H|\,,
\end{align*}
i.e. $H=G$. This implies that all terms with $H\ne G$ necessarily satisfy
\[
|H^1_{\mathcal{L}(H\cap T),\inv_H^\downarrow}(K,(H\cap T)(\pi);X)| = o(X^{1/a_\inv(\mathcal{L})} (\log X)^{b_{\inv}(\mathcal{L})-1-\epsilon})\,.
\]
Thus it follows that
\begin{align*}
&|\{f\in H^1_{\mathcal{L},\inv}(K,T(\pi);X) : f*\pi \text{ surjective}\}|\\
& = |H^1_{\mathcal{L},\inv}(K,T(\pi);X)| - \sum_{\substack{H< G\\HT=G}}|\{f\in H^1_{\mathcal{L},\inv}(K,T(\pi);X) : (f*\pi)(G_K)\subset H \}|\\
&= |H^1_{\mathcal{L},\inv}(K,T(\pi);X)| + o(X^{1/a_\inv(\mathcal{L})} (\log X)^{b_{\inv}(\mathcal{L})-1-\epsilon}).
\end{align*}
Theorem \ref{thm:asymptoticWiles} implies the little-oh term grows strictly slower that $|H^1_{\mathcal{L},\inv}(K,T(\pi);X)|$, concluding the proof.
\end{proof}

\subsection{Lower Bounds for Malle's Conjecture}

We are now ready to prove the lower bounds for Malle's conjecture proper detailed in the introduction.
\begin{proof}[Proof of Corollary \ref{cor:lowerbound}]
If there exists a $G$-extension given by $\pi$, then we fix the subextension $L/K$ fixed by $T$. Then Corollary \ref{cor:alpha} implies
\[
N(K,G;X) \gg N(L/K,T\normal G;X) \gg X^{1/a(T)}(\log X)^{b(K,T(\pi))-1}.
\]
If there exists a $t\in T$ with $\ind(t)=a(G)$, then minimality implies $a(G)=a(T)$ and $b(K,T(\pi))\ge 1$ implies
\[
N(K,G;X) \gg X^{1/a(G)}.
\]
If $A(G) \subset T$, then $A(G)=A(T)$ and Turkelli's modification to Malle's conjecture asserts that
\begin{align*}
B(K,G) &= \max_{\substack{N\normal G\\N\normal T\\ a(N)=a(T)} }\max_{\substack{\varphi:G_K\rightarrow G\\ \varphi \equiv \pi \mod N} }b(K,N(\varphi))\\
&=\max_{\substack{\pi:G_K\rightarrow G\\ \pi(G_K)T=G} }b(K,\langle A(T)\rangle(\pi))\,,
\end{align*}
noting that $T$ is abelian, so the number of orbits is maximal when the conjugacy classes of $A(T)$ are of minimal size, i.e. are trivial in the abelian subgroup $\langle A(T)\rangle\le T$. For abelian groups, the conjugacy classes of $\langle A(T)\rangle$ and $T$ are the same, which implies
\begin{align*}
B(K,G) = \max_{\substack{\pi:G_K\rightarrow G\\ \pi(G_K)T=G} }b(K,T(\pi))\,.
\end{align*}
Choosing some $\pi$ that achieves this maximum corresponding to a $B$-extension $L/K$ implies $b(K,T(\pi))=B(K,G)$ and
\[
N(K,G;X) \gg N(L/K,T\normal G;X) \gg X^{1/a(G)}(\log X)^{B(K,G)-1}.
\]
\end{proof}
Corollary \ref{cor:solvable} then follows immediately by choosing $g\in G$ which commutes with its conjugates such that $\ind(g)=a$ and setting $T=\langle c_x(g) : x\in G\rangle$. This is an abelian group with $a=a(T)$. The result then follows from Corollary \ref{cor:lowerbound}.

\newpage

\appendix

\section{Wright's Methods and the Fundamental Class}\label{app}

An alternative approach the the question of counting $(T\normal G)$-towers was described to me by Melanie Matchett Wood, as an adaptation of Wright's more classical proof of Malle's original conjecture for abelian groups using class field theory. In this appendix we summarize this argument, and discuss the major different benefits of this approach versus the approach using Wiles' theorem in the main body of this paper.

\subsection{Counting towers of number fields}

The Galois groups of a tower of number fields $F/L/K$ fit into a short exact sequence
\[
\begin{tikzcd}
1 \rar & \Gal(F/L) \rar & \Gal(F/K) \rar & \Gal(L/K) \rar & 1.
\end{tikzcd}
\]
If $F/L/K$ is a $(T\normal G)$-tower with $T$ abelian and $B:=G/T$, then this short exact sequence belongs to some extension class $v\in H^2(B,T)$ using the description of $H^2$ as the group of extensions of $B$ by $T$ (we recall that when $T$ is abelian, the action of $G$ on $T$ by conjugation factors through $B$).\cite[Theorem 11.5]{cassels-frohlich2010} states that there exists a fundamental class $u_{L/K}\in H^2(B,J_L)$, where $J_L$ is the idele class group, such that for any $B$-invariant homomorphism $\phi\in \Hom_B(J_L,T)$ defining an abelian extension $F/L$ the extension class of $F/L/K$ is given by $\phi_*(u_{L/K})\in H^2(B,T)$.

Let $S$ be a $B$-invariant set of places of $L$ which generate the idele class group, so that by class field theory there is a short exact sequence
\[
\begin{tikzcd}
1 \rar & \mathcal{O}_S^\times \rar & \prod_{p\in S} L_p^\times \times \prod_{p\not\in S} \mathcal{O}_{L_p}^{\times} \rar & J_L \rar & 1.
\end{tikzcd}
\]
Fix a lift $u\in C^2(B,J_L)$ of $u_{L/K}$ and define the homomorphism
\[
\Psi:\Hom_B\left(\prod_{p\in S} L_p^\times \times \prod_{p\not\in S} \mathcal{O}_{L_p}^{\times},T\right) \rightarrow \Hom(\mathcal{O}_S^\times,T) \times C^2(B,T)
\]
by restriction in the first coordinate and $\phi\mapsto \phi_*(u)$ in the second coordinate. The generating function of towers $F/L/K$ with extension class $v$ is then given by
\[
\sum_{\substack{w\in C^2(B,T)\\ w\equiv v}}\left(\sum_{\substack{\phi\in \Hom_B\left(\prod_{p\in S} L_p^\times \times \prod_{p\not\in S} \mathcal{O}_{L_p}^{\times},T\right)\\ \Psi(\phi)=(1,w)}} \mathcal{N}_{L/\Q}(\disc(\phi))^{-s}\right).
\]
The group $A=\Hom(\mathcal{O}_S^\times,T) \times C^2(B,T)$ is a finite abelian group, so that we can take a sum over characters as is done in \cite{wright1989,wood2009} to get
\[
\sum_{\substack{w\in C^2(B,T)\\ w\equiv v}}\sum_{\chi\in A^{\vee}}\left(\sum_{\phi\in \Hom_B\left(\prod_{p\in S} L_p^\times \times \prod_{p\not\in S} \mathcal{O}_{L_p}^{\times},T\right)} \chi(\Psi(\phi)(1,w)^{-1})\mathcal{N}_{L/\Q}(\disc(\phi))^{-s}\right).
\]
The multiplicativity of $\chi$, $\disc$, and the fact that
\[
\Hom_B\left(\prod_{p\in S} L_p^\times \times \prod_{p\not\in S} \mathcal{O}_{L_p}^{\times},T\right) = \prod_{\substack{\ell\in \overline{S}\\\text{places of }K}} \Hom_B\left(\prod_{p\mid \ell}L_p^{\times},T\right) \prod_{\substack{\ell\not\in \overline{S}\\\text{places of }K}}\Hom_B\left(\prod_{p\mid\ell} \mathcal{O}_{L_p}^{\times},T\right)
\]
Implies that this inner-most sum factors as an Euler product over the places $\ell$ of $K$
\begin{align*}
Q(w,\chi,s) =& \prod_{\ell\in \overline{S}} \left(\sum_{\phi_\ell\in\Hom_B\left(\prod_{p\mid \ell} L_p^\times,T\right)} \chi(\Psi(\phi_\ell)(1,w)^{-1})\mathcal{N}_{L/\Q}(\disc(\phi))^{-s}\right)\\
&\cdot\prod_{\ell\not\in \overline{S}} \left(\sum_{\phi_\ell\in\Hom_B\left(\prod_{p\mid \ell} \mathcal{O}_{L_p}^\times,T\right)} \chi(\Psi(\phi_\ell)(1,w)^{-1})\mathcal{N}_{L/\Q}(\disc(\phi))^{-s}\right).
\end{align*}
By applying Proposition \ref{prop:GeneralDirichlet} in the same way as the main body of the paper, we find that each Euler product has a rightmost pole on the real line (where checking that the Euler factors are Frobenian is similar to the proof of Theorem \ref{thm:MBlocal}). It then suffices to compute the location and order of the rightmost poles of each Euler factor using Proposition \ref{prop:GeneralDirichlet} and to show that the poles do not cancel when added together. The computation of the location and order of the poles is similar to the computations in \cite{wright1989,wood2009}, so we omit the process from this appendix.

In order to show that the poles do not cancel, it suffices to find a lower bound for the counting function of the same order of magnitude. Indeed, for each $w\in C^2(B,T)$ consider a $B$-invariant set of places $S'$ containing $S$ such that $\Psi$ restricted to
\[
\Hom_B\left(\prod_{p\in S} L_p^\times \times \prod_{p\in S'\setminus S} \mathcal{O}_{L_p}^{\times},T\right)
\]
is surjective. This is guaranteed to exist by $\Psi$ having a finite range. \emph{If we assume} that $(1,w)$ is in the image of $\Psi$ (this is equivalent to assuming the existence of a solution to the embedding problem), then for every $\phi$ we choose some
\[
\gamma_{\phi,S'} \in \Hom_B\left(\prod_{p\in S} L_p^\times \times \prod_{p\in S'\setminus S} \mathcal{O}_{L_p}^{\times},T\right)
\]
such that $\Psi(\gamma_{\phi,S'}) = (1,w)\Psi(\phi)^{-1}$. Consider the subset of homomorphisms given by
\[
U:=\{\phi\gamma_{\phi,S'} : \phi \in \Hom_B\left(\prod_{p\in S} L_p^\times \times \prod_{p\not\in S} \mathcal{O}_{L_p}^{\times},T\right).
\]
Then by construction $\Psi(U)=\{(1,w)\}$. Moreover, if we modify the discriminant so that
\[
\nu_p(\disc_{S'}(\phi)) = \begin{cases}
\nu_p(\disc(\phi)) & p\not\in S'\\
\max_{\gamma\in \prod_{p\in S'}\Hom(L_p^\times,T)} \nu_p(\disc(\gamma)) & p\in S'
\end{cases}
\]
it follows that $\disc_{S'}(\phi\gamma_{\phi,S'})=\disc_{S'}(\phi)$ the counting function
\[
\#\{\kappa \in U: \mathcal{N}_{L/\Q}(\disc_{S'}(\kappa))<X\}
\]
is a lower bound for the number of towers with bounded discriminant. The corresponding generating function is a single Euler product given by
\[
\prod_{\ell\not\in \overline{S'}} \left(\sum_{\phi_\ell\in\Hom_B\left(\prod_{p\mid \ell} \mathcal{O}_{L_p}^\times,T\right)}\mathcal{N}_{L/\Q}(\disc(\phi))^{-s}\right).
\]
Applying Proposition \ref{prop:GeneralDirichlet} will produce a single rightmost pole, which gives this lower bound the correct order of magnitude.

\subsection{Comparing the methods}

This adaptation of Wright's proof by appealing to the fundamental class has some major benefits:
\begin{enumerate}
\item{It does not suffer from the same obstruction as Wiles' theorem (namely that the local conditions $L_p\le H^1(K_p,T)$ must be \emph{subgroups} instead of subsets), which as a result produce the asymptotic main term on the nose for the discriminant ordering.}

\item{The method is already well established in \cite{wright1989,wood2009} and shows that class field theory is still behind the scenes of these results.}
\end{enumerate}

The methods in the main body of the paper using Wiles' theorem, in addition to being a new approach to number field counting, are highlighted in part due to the ease of which they produce general results for restricted local conditions (this is already built into Wiles' theorem!) and in part for the clarity they provide for future work.

In Wright's original paper, the sum of characters $\chi\in A^{\vee}$ is evaluated by choosing roots of unity and forcing a noncanonical isomorphism between $A^{\vee}$ and some product $\prod_{i} \mathcal{O}_S^{\times}/(\mathcal{O}_S^{\times})^{n_i}$ in order to relate the characters $\chi$ to Dirichlet characters. While this is sufficient to produce the asymptotic main term of the counting function, it obscures the nature of the finite sum of Euler products as the sum over a dual object or as the sum over a class group-like object. Proposition \ref{prop:Eulerproduct} shows that the finite sum can be interpreted as a finite sum over the dual Selmer group, $H^1_{\mathcal{L}(0)^*}(K,T^*)$. Looking forward towards proving uniformity in step 2, the dual Selmer group can be understood as a class group object over $L(\mu_{|T|})$, which readily makes it clear how the length of this sum depends on the intermediate extension $L/K$ (in particular, we get the trivial bound $\mathcal{N}_{K/\Q}(\disc(L(\mu_{|T|})/K))^{1/2+\epsilon}$ on the length of the sum).

We also remark that the Wiles' theorem approach highlights the new question of counting cohomology groups with nontrivial actions, $H^1(K,T)$, as related to counting towers of number fields. We saw in Section \ref{sec:towers} has a similar local structure so as to generalize Malle's conjecture to this setting, and aside from the application to towers this question becomes an interesting generalization in its own right. The sketch in this appendix does not require counting coclasses in $H^1$, and so misses out on an entire new generalization of number field counting.

\bibliographystyle{alpha}
\bibliography{Statistics of the First Galois Cohomology Group.bbl}

\end{document}